\documentclass[11pt]{article}

\usepackage{amssymb}   
\usepackage{amsthm}    
\usepackage{amsmath}   
\usepackage{stmaryrd}  
\usepackage{titletoc}  
\usepackage{mathrsfs}  
\usepackage{graphicx}

\newlength{\defbaselineskip}
\newcommand{\setlinespacing}[1]%
           {\setlength{\baselineskip}{#1 \defbaselineskip}}

\theoremstyle{plain}
\newtheorem{thm}{Theorem}[section]

\newtheorem{lem}[thm]{Lemma}
\newtheorem{prop}[thm]{Proposition}

\theoremstyle{definition}
\newtheorem{defn}{Definition}[section]

\newtheorem{rmk}{Remark}[section]

\newcommand{\eps}{\varepsilon}

\DeclareMathOperator*{\esssup}{ess\,sup}

\newcommand{\cL}{\mathcal{L}}

\newcommand{\cB}{\mathcal{B}}

\newcommand{\cO}{\mathcal{O}}
\newcommand{\cA}{\mathcal{A}}

\newcommand{\bP}{\mathbb{P}}
\newcommand{\bR}{\mathbb{R}}
\newcommand{\bN}{\mathbb{N}}

\newcommand{\sF}{\mathscr{F}}

\newcommand{\la}{\langle}
\newcommand{\ra}{\rangle}

\newcommand{\tbf}{\textbf}

\makeatletter\@addtoreset{equation}{section} \makeatother
\begin{document}

\title{On Backward Doubly Stochastic Differential  Evolutionary System}
\author{Jinniao Qiu \footnotemark[2]~~~~ and ~~~~
Shanjian Tang\footnotemark[2]~\footnotemark[3]}

\footnotetext[1]{Supported by NSFC Grant \#10325101, by Basic Research Program of China (973 Program)  Grant \# 2007CB814904, by the Science
Foundation of the Ministry of Education of China Grant \#200900071110001, and by WCU (World Class University) Program through the Korea
Science and Engineering Foundation funded by the Ministry of Education, Science and Technology (R31-2009-000-20007).}

\footnotetext[2]{Department of Finance and Control Sciences, School of Mathematical Sciences, Fudan University, Shanghai 200433, China.
\textit{E-mail}: \texttt{qiujinn@gmail.com} (Jinniao Qiu), \texttt{sjtang@fudan.edu.cn} (Shanjian Tang).}

\footnotetext[3]{Graduate Department of Financial Engineering, Ajou University, San 5, Woncheon-dong, Yeongtong-gu, Suwon, 443-749, Korea.}

\maketitle

\begin{abstract}
In this paper, we are concerned with  backward doubly stochastic differential evolutionary systems (BDSDESs for short). By using a variational approach based on the monotone operator theory, we prove the existence and uniqueness of the solutions for BDSDESs. We also establish an It\^o formula for the Banach space-valued BDSDESs.
\end{abstract}

AMS Subject Classification: 60H15; 35R60

Keywords: Gelfand triple, Monotone operator, Backward doubly stochastic evolutionary system, Backward doubly stochastic differential equation, Backward doubly stochastic partial differential equation

\section{Introduction}
 Starting from Bismut's pioneering work \cite{Bismut_78,Bismut_78_2} and
Pardoux and Peng's seminal work \cite{ParPeng_90}, the theory of backward
 stochastic differential equations (BSDEs) is rather complete (for instance, see \cite{Hu_2002,DelbaenTang_10,Karoui_Peng_Quenez}). As a natural generalization of BSDEs,
  backward stochastic partial differential equations (BSPDEs) arise
in many applications of probability theory and stochastic processes, for
instance in the optimal control of processes with incomplete information,
as an adjoint equation of the Duncan-Mortensen-Zakai filtration equation
(for instance, see \cite{Bensousan_83,DuQiuTang10,Hu_Ma_Yong02,Tang_98,Zhou_92}), and
naturally in the dynamic programming theory fully nonlinear BSPDEs as the
so-called  backward stochastic Hamilton-Jacobi-Bellman equations, are also
introduced in the study of  non-Markovian control problems (see Peng
\cite{Peng_92} and Englezos and Karatzas \cite{EnglezosKaratzas09}).

 In this work, we consider
the following backward doubly stochastic differential evolutionary system:
\begin{equation}\label{bdsdes}
  \left\{\begin{array}{l}
    \begin{split}
      -du(t)=F(t,u(t),v(t))\ dt+J(t,u(t),v(t))\ d\overleftarrow{B}_t-v(t)\ dW_t,~~~t\in[0,T],
    \end{split}\\
    \begin{split}
      u(T)=G,
    \end{split}
  \end{array}\right.
\end{equation}
which are first introduced by Pardoux and Peng \cite{PadouxPeng94BDSDE} as
backward doubly SDEs (BDSDEs, for short) to give a probabilistic representation for certain systems of quasilinear stochastic partial differential equations (SPDEs, for short). Similar arguments to Tang  \cite{TangBSPDEinOsaka2010} yield that this class of BDSDES includes SDEs, BSDEs, BDSDEs, SPDEs, backward stochastic partial differential equations (BSPDEs, for short) and backward doubly SPDEs (BDSPDEs, for short) as particular cases.

  On account of the connections between BDSDEs and SPDEs, many important results for SPDEs have been obtained: Buckdahn and Ma  \cite{BuckdahnMa2001,BuckdahnMa2002} established a stochastic viscosity solution theory for SPDEs;
 through investigations into a class of generalized BDSDEs, Boufoussi, Casteren and Mrhardy  \cite{Boufoussial2007} gave a probability representation for the stochastic viscosity solution of SPDEs with nonlinear Neumann boundary conditions; Zhang and Zhao  \cite{ZhangQiZhaoHD2007} used the extended Feymann-Kac formula to study the stationary solutions of SPDEs; Ichihara \cite{Ichihara04}  discussed the homogenization problem for SPDEs of Zakai type through the BDSDE theory; Matoussi and Stoica proved the existence and uniqueness result for the obstacle problem of quasilinear parabolic stochastic PDEs. Recently, Han, Peng and Wu  \cite{HanPengWu2010} established a Pontryagin type maximum principle for the optimal control problems with the state process driven by BDSDEs. It is worth noting that all the BDSDESs involved in the above results are finite dimensional. For the infinite dimensional case, BDSPDEs are first introduced and studied in Tang \cite{TangBSPDEinOsaka2010} by using the method of stochastic flows, while the generalized solution theory for BDSPDEs is  blank.

By using a variational approach based on the monotone operator theory, 
we investigate the BDSDESs and  prove the existence and uniqueness of the solutions for BDSDESs. The results seems to be new both for the finite dimensional case (BDSDEs) and the infinite dimensional case. Moreover, our results are also expected to extend the results of the previous paragraph to the infinite dimensional cases, i.e., we may extend Feymann-Kac formula and establish the stochastic viscosity theory for SPDEs on Hilbert spaces, construct stationary solutions for SPDEs on infinite dimensional spaces and investigate the optimal control problems with state processes driven by infinite dimensional BDSDESs. As an application to quasi-linear BDSPDEs, we get a more general existence and uniqueness result both for BDSPDEs and BSPDEs, which fills up the gap of the generalized solution theory for BDSPDEs. For the Banach space-valued BDSDESs, we also prove an It\^o formula, which plays an equally important role as that for SPDEs (for instance, see \cite{Krylov_Rozovskii81,PrevotRockner2007,RenRocknerWang2007}).

Our paper is organized as follows. In the next section, we set notations, hypotheses, and  the notion of the solution to BDSDES \eqref{bdsdes} and list the main theorem. In Section 3, we prepare several auxiliary results, including a generalized It\^o formula for the Banach space-valued BDSDESs, and a useful lemma on the weak convergence which is proved through a variational approach on basis of the monotone operator theory and will be used frequently in the following context. In section 4, by using the Galerkin approximation, we prove our main theorem first for the finite dimensional case and then the infinite dimensional case. In section 5, we apply our results to several examples. Section 6 is the appendix in which we prove our It\^o formula for the Banach space-valued BDSDESs.

\section{Preliminaries}
Let $V$ be a real reflexive and separable Banach space, and $H$ a real
separable Hilbert space. The norm in $V$ is denoted by $\|\cdot\|_V,$ and
the inner product and norm in $H$ is denoted by $\langle \cdot,\ \cdot\rangle $ and
$\|\cdot\|$ respectively. In this work we always assume that $V$ is dense,
and continuously imbedded in $H$. Thus, the dual space $H'$ is also
continuously imbedded in $V'$ which is the dual space of $V$. Simply, we denote the
above framework by
\begin{equation*}
  V\hookrightarrow H \cong H'\hookrightarrow V'.
\end{equation*}

We denote by $\|\cdot\|_*$ the norm in $V'$. The dual product between $V$
and $V'$ is denoted by $_{V'}\langle \cdot,\ \cdot \rangle _V$. Since it follows that
$$_{V'}\langle \varphi,\ \phi\rangle _V\ \,=\ \,\langle \varphi,\ \phi\rangle ,\ \varphi\in H,\phi\in V,$$
we shall still denote the dual product between $V$ and $V'$ by
$\langle \cdot,\ \cdot\rangle $ with a little notational confusion. $(V,H,V')$ is called a Gelfand triple.

 Fix  a finite time $T>0$. Let
$(\Omega,\sF,\bP)$ be a complete filtered probability space on which are
defined two mutually independent cylindrical  Wiener processes
$W=\{W_t:t\in[0,T]\}$ and $B=\{B_t: t\in[0,T] \}$ taking values on separable Hilbert
spaces $(U_1,\langle ,\ \rangle _{U_1},\|\cdot\|_{U_1})$ and $(U_2,\langle ,\ \rangle _{U_2},\|\cdot\|_{U_2})$ respectively. Denote by $(L(U_i,H),\langle,\ \rangle_i,\|\cdot\|_i)$ the separable  Hilbert
space of all Hilbert-Schmidt operators from $U_i$ to $H$, $i=1,2$.
 Denote by $\mathcal {N}$ the set of all the $\bP$-null sets in $\sF$.
For each $t\in [0,T]$, define
\begin{equation*}\label{filtration}
    \sF_t:=\sF_t^W\vee\sF_{t,T}^B
\end{equation*}
where for any process $\eta$, $\sF_{s,t}^{\eta}:=
\sigma\{\eta_r-\eta_s:s\leq r\leq t\}\vee\mathcal{N}$ and
$\sF_t^{\eta}:=\sF_{0,t}^{\eta}$. Note that as a collection of $\sigma$-algebras, $\{\sF_t,t\in[0,T]\}$ is not a filtration, since it is neither increasing nor decreasing. Consider
 BDSDES \eqref{bdsdes} and write it into the following integral form:
\begin{equation}\label{bdsdes-integral}
    \begin{split}
      u(t)=\ &G+\int_t^T F(s,u(s),v(s))\ ds+\int_t^T J(s,u(s),v(s))\ d\overleftarrow{B}_s \\
      &-\int_t^Tv(s)\ dW_s
      ,~t\in[0,T]
    \end{split}
\end{equation}
where for any $(t,\varphi,\phi)\in [0,T]\times V\times L(U_1,H),$
$$F(\cdot,t,\varphi,\phi):~\Omega\rightarrow V' \textrm{ and }
 J(\cdot,t,\varphi,\phi):~\Omega\rightarrow L(U_2,H)
 $$
are $\sF_t$-measurable.
Moreover, in \eqref{bdsdes} and \eqref{bdsdes-integral} the integral  with respect to $\{B_t\}$ is a $backward$ It\^o integral and the integral   with respect to $\{W_t\}$ is a standard It\^o
integral (c.f. \cite{NualaPardou98}).

For any $p,q\in [1,\infty] $ and any real separable Banach space $(U,\|\cdot\|_U)$, denote by
$M^{p,q}(0,T;U)$ the totality of $\varphi\in L^p(\Omega,\sF,L^q([0,T],\cB([0,T]),U))$
with
$$\|\varphi\|_{M^{p,q}(0,T;U)}:=\|\varphi\|_{L^p(\Omega,\sF,L^q([0,T],\cB([0,T]),U))}$$
such that $\varphi_t$ is
$\sF_t$-measurable, for a.e. $t\in [0,T]$.
For simplicity, set
$$M^p(0,T;U):=M^{p,p}(0,T;U).$$
 For $r\in [1,\infty)$ we
denote by $S^r(0,T;U)$ the totality of $\phi\in L^r(\Omega,\sF,C([0,T],U))$
such
that $\phi_t$ is $\sF_t$-measurable, for any
$t\in [0,T]$. Define
$$\|\phi\|_{S^r(0,T;U)}:=\left\{ E\left[ \sup_{t\in[0,T]}\|\phi(t)\|_U^r  \right] \right\}^{1/r},\ \phi\in S^r(0,T;U).
$$
  All the spaces defined above are complete.

 Moreover,
Letting $\tau$ ($0\leq \tau\leq T$) be a stopping time with respect to the backward filtration  $\{\sF_{t,T}^B,t\in [0,T]\}$, define
$$M^p(\tau,T;U):=\{1_{ [\tau,T]}u:\ u\in M^p(0,T;U)  \}$$  equipped with the norm
$\|u\|_{M^p(\tau,T;U)}=\|1_{ [\tau,T]}u\|_{M^p(0,T;U)}$
and in a similar way, we define $S^p(\tau,T;U)$ and $M^{p,2}(\tau,T;U)$.

 For simplicity, we always denote $M^p(\tau,T;\bR)$ ($S^p(\tau,T;\bR)$ and $M^{p,q}(\tau,T;\bR)$, respectively) by $M^p(\tau,T)$
($S^p(\tau,T)$ and $M^{p,q}(\tau,T)$, respectively).

By convention, we always treat elements of spaces like $ M^p(0,T;U)$
defined above
 as functions rather than distributions or classes of equivalent functions,
and if we know that a function of this class has a modification with better
properties, then we always consider this modification. For example, if
$u\in M^p(0,T; V)$ and $u$ has a modification lying in $S^q(0,T;H)$, we
always adopt the treatment $u\in M^p(0,T;V)\cap S^q(0,T;H)$.


 \medskip
Consider our BDSDES \eqref{bdsdes}. We define the following assumptions.

\medskip
There exist constants $1>\delta>0,\alpha>0,q>1,\alpha_1,K,K_1,\beta\geq 0$ and
    a nonnegative real-valued process
    $\varsigma\in M^1(0,T)$ such that the following conditions
    hold for all $v,v_1,v_2\in V,$ $\phi,\phi_1,\phi_2\in L(U_1,H)$ and
    $(\omega,t)\in \Omega\times [0,T]$.

\bigskip\medskip
   $({\mathcal A} 1)$ \it (Hemicontinuity)
   The map
    $s\mapsto \langle F(t,v_1+sv_2,\phi),\ v\rangle $ is continuous  on $\bR$.\rm

\medskip
   $({\mathcal A}2)$ \it  (Monotonicity)
\begin{equation*}
\begin{split}
    &2\langle F(t,v_1,\phi_1)-F(t,v_2,\phi_2),\ v_1-v_2\rangle +
    \|J(t,v_1,\phi_1)-J(t,v_2,\phi_2)\|_2^2\\
    &\leq
    K_1\|v_1-v_2\|^2+\delta\|\phi_1-\phi_2\|_1^2;
\end{split}
\end{equation*}\rm

\medskip
   $({\mathcal A} 3)$ \it (Coercivity)
    $$2\langle F(t,v,\phi),\ v\rangle +\|J(t,v,\phi)\|_2^2+\alpha\|v\|_V^q
    \leq
        \delta \|\phi\|_1^2+K\|v\|^2+\varsigma(t);$$\rm

\medskip
   $({\mathcal A} 4)$ \it (Growth)
  \begin{equation*}
    \begin{split}
        &\|F(t,v,\phi)\|_*^{q'}
        \leq
            \left[\varsigma(t)+K\left(\|v\|_V^q+\|v\|^2+\|\phi\|_1^2\right)\right]\left(1+\|v\|^{\beta}\right),\\
        &\|J(t,v,\phi)\|_2^2
            \leq K\left(\varsigma(t)+\|v\|^q_V+\|v\|^2+\|\phi\|_1^2\right),
            \quad \frac{1}{q'}+\frac{1}{q}=1;
    \end{split}
  \end{equation*}\rm

\medskip
   $({\mathcal A} 5)$ \it (Lipchitz Continuity)
    \begin{equation*}
        \begin{split}
        &\|F(t,v,\phi_1)-F(t,v,\phi_2)\|_*\leq K\|\phi_1-\phi_2\|_1,\\
        &\|J(t,v_1,\phi)-J(t,v_2,\phi)\|_2 \leq K\|v_1-v_2\|_V;
        \end{split}
    \end{equation*}\rm

\medskip
   $({\mathcal A} 6)$ \it
 $$J(t,v,\phi)J^*(t,v,\phi)\leq \phi\phi^*+K(\|J(t,0,0)\|_2^2+\|v\|^2)I+\alpha_1\|v\|_V^q\wedge\|v\|_V^2
 I$$
 where  $J^*$ ($\phi^*$, respectively)
 denotes the adjoint transformation of $J$ ($\phi$, respectively) and $I$ is the identity operator on $H$.\rm

\begin{rmk}
	Actually,  we can deduce from $(\mathcal{A}2)$ and $(\mathcal{A}3)$ that
  for all $\phi_1,\phi_2\in L(U_1,H)$, $v\in V$ and $(\omega,t)\in \Omega\times[0,T]$, there hold
  $$\max\{\|J(t,0,0)\|_2^2,\,\|F(t,0,0)\|_*^{q'}\}\leq \varsigma(t)\textrm{ and } \|J(t,v,\phi_1)-J(t,v,\phi_2)\|_2^2\leq \delta\|\phi_1-\phi_2\|_1^2.$$
\end{rmk}

\begin{rmk}
  In view of $(\mathcal{A}4)$ and $(\mathcal{A}6)$, we see that the function $\|J(t,\cdot,\phi)\|_*^q$ is defined on $V$ and dominated by the norm $\|\cdot\|_V^q$ in some sense. This property goes beyond the calculations of \cite{Krylov_Rozovskii81,LiuwRockner2010,Pardoux1979,PrevotRockner2007,RenRocknerWang2007,ZhangXich2009}. Moreover, if $J(t,v,\phi)J^*(t,v,\phi)$ does not depend on $\|\phi\|_1$ or $\|v\|_V$, the assumption $(\mathcal{A}6)$ is not necessary in our work.  In addition, as $J(t,v,\phi)$ is Lipchitz continuous with respect to $v$ on $V$, it seems not so strange that $J(t,v,\phi)J^*(t,v,\phi)$ is dominated by $\|v\|_V^2$.
\end{rmk}

\begin{defn}\label{definition solution}
  We say a pair of $V\times L(U_1,H)$-valued processes $(u,v)$
  is a solution of the backward doubly stochastic differential evolutionary system
  \eqref{bdsdes}
  if
  	$$(u,v)\in \left(M^{pq/2,q}(0,T;V)\cap S^p(0,T;H)\right)\times M^{p,2}(0,T;L(U_1,H)),\textrm{ for some }p\geq 2,q>1$$
  and \eqref{bdsdes} holds in the weak sense (called in the distributional sense as well), i.e.
   for any $\varphi\in V$ there holds almost surely
  \begin{equation}
    \begin{split}
      \langle \varphi,\ u(t)\rangle \ =\
      &\langle \varphi,\ G\rangle +\int_t^T\! \langle F(s,u(s),v(s),\ \varphi\rangle \ ds
                            -\int_t^T\langle \varphi,\ v(s)dW_s\rangle \\
      &+\int_t^T\langle \varphi,\ J(s,u(s),v(s))d\overleftarrow{B}_s\rangle
        ,\ \forall t\in [0,T].
    \end{split}
  \end{equation}
\end{defn}

Now we show our main result as the following theorem:
\begin{thm}\label{thm main}
  Suppose  assumptions $(\mathcal{A}1)$-$(\mathcal{A}5)$ hold. Let $0\leq \alpha_1(p-2)<\alpha$,
  and $\|F(\cdot,0,0)\|_*^{q'},\varsigma\in M^{ p/2,1}(0,T)$
  for some $p\geq \beta+2$. Moreover, if $p>2$, we assume $(\mathcal{A}6)$ holds.
  Then for any $G\in L^p(\Omega,\sF_T,H)$, BDSDES \eqref{bdsdes} admits a unique
  solution
$$(u,v)\in \left(M^{pq/2,q}(0,T;V)\cap S^p(0,T;H)\right)\times M^{p,2}(0,T;L(U_1,H))$$
such that
\begin{equation}\label{main thm estimate}
\begin{split}
    &\|u\|_{S^p(0,T;H)}+\|u\|^{q/2}_{M^{pq/2,q}(0,T;V)}+\|v\|_{M^{p,2}(0,T;L(U_1,H))}\\
    \leq\ & C \left\{  \|G\|_{L^p(\Omega,\sF_T,H)}+\|\varsigma\|^{1/2}_{M^{p/2,1}(0,T)}   \right\}
\end{split}
\end{equation}
where $C$ is a constant depending on $T,K,q,p,\delta,\beta,\alpha_1$ and
$\alpha$.
\end{thm}
%
%

Here, we point out that we always denote by $C>0$ a constant which may vary from line to line and moreover, we denote by $C(a_1,a_2,\cdots)$ a constant which depends on the variables $a_1,a_2,\cdots$ just like the one appearing in the following typical inequality
$$ ab\leq \eps a^2+C(\eps) b^2, \eps>0,a,b\in \bR.$$
\section{Auxiliary results}

First, we give a useful lemma with the sketch of its proof.
\begin{lem}\label{lem in prelimilary}
 For any given $p\geq 1,q,d>1,r\geq 2$ and separable reflexive Banach spaces $U$ and $\bar{U}$, with $U$ continuously and densely embedded into $\bar{U}$, we assert that

  (i) $M^p(0,T;U)$, $S^p(0,T; U)$ and $M^{q,d}(0,T;L(U_i,H)),i=1,2$ are all separable Banach spaces, and moreover, $M^q(0,T;U)$ and $M^{q,d}(0,T;L(U_i,H)),i=1,2$ are reflexive;

  (ii) let $\{u_n,n\in\bN\}$ converge weakly to $u$ in $M^p(0,T;U)$ and to $\bar{u}$ in $M^{q,d}(0,T;\bar{U})$, then
  $\bar{u}(\omega,t)=u(\omega,t)$ for $\mathbb{P}\otimes dt$-almost all $(\omega,t)\in \Omega\times [0,T]$;

  (iii) define two linear operators
  \begin{equation}
  \begin{split}
    \mathcal{I}(f): & =\int_{\cdot}^T f(s)\ ds,~~f\in M^q(0,T;U);\\
    \mathcal{J}(h): & =\int^T_{\cdot} h(s)\ dW_s,~~h
    \in M^{q,2}(0,T;L(U_1,H));
  \end{split}
  \end{equation}
then the linear operators $\mathcal{I}$ and $\mathcal{J}$ are continuous from $M^q(0,T;U)$ to itself and from $M^{q,2}(0,T;L(U_1,H))$ to $M^{q,2}(0,T;H)$ respectively, and moreover, they both are continuous with respect to the corresponding weak topologies;

    (iv) letting $u_n$, $f_n$, $h_n$ and $z_n$ converge weakly to $u$, $f$, $h$, and $z$ in spaces $M^p(0,T;H)$,
    $M^q(0,T;V')$, $M^{r,2}(0,T;L(U_2,H))$ and $M^{r,2}(0,T;L(U_1,H))$ respectively, then we conclude from
    \begin{equation}
      \begin{split}
        &\lim_{n\rightarrow \infty}\|G^n-G\|_{L^d(\Omega,\sF_T,H)}=0,\\
        &u_n(t)=G^n+\int_t^Tf_n(s)\ ds+\int_t^Th_n(s)\ d \overleftarrow{B}_s-\int_t^Tz_n(s)\ dW_s\\
       \textrm{and  }&
        \bar{u}(t):=G+\int_t^Tf(s)\ ds+\int_t^Th(s)\ d\overleftarrow{B}_s-\int_t^Tz(s) \ dW_s
      \end{split}
    \end{equation}
    that $u(\omega,t)=\bar{u}(\omega,t)$ for $\mathbb{P}\otimes dt$-almost all $(\omega,t)\in \Omega\times[0,T]$.
\end{lem}
\begin{proof}
  (i) is obvious. From the definition of weak convergence, it follows that
  $$\int_{\Omega\times[0,T]} 1_{\{ A \}}(\omega,s)f(u(\omega,s)-\bar{u}(\omega,s) )\ \mathbb{P}(d\omega)ds=0,\ \forall A\in \sF\otimes\cB([0,T]),\,\forall f\in \bar{U}',$$
    which implies (ii). As for (iii), it follows from
       $$\|\mathcal{I}(f)\|_{M^q(0,T;U)}^q
       =E\left[\int_0^T\Big\|\int^T_t f(s)\ ds\Big\|_U^q  \  dt\right]
       \leq T^q E\left[\int_0^T\|f(s)\|_U^q \   ds\right]$$
and
   \begin{equation}
     \begin{split}
       \|\mathcal{J}(h)\|_{M^{q,2}(0,T;H)}^q
       =&\ E\left[\left(\int_0^T\|\int^T_th(s)\ dW_s\|^2\ dt\right)^{q/2}\right]\\
        \leq&\
	T^{q/2} E\left[\sup_{t\in [0,T]}\|\int^T_th(s)\ dW_s \|^q\right]\\
        \leq&\
         CE\left[ \left(\int_0^T\|h(s)\|_1^2\ ds\right)^{q/2}\right].
     \end{split}
   \end{equation}
   Finally, (iv) can be deduced from the above assertions (i), (ii) and (iii).
\end{proof}

As in the theory on the forward stochastic evolutionary systems (c.f. \cite{Krylov_Rozovskii81,RenRocknerWang2007}), the following It\^o formula plays a crucial role in the proof of our main result Theorem \ref{thm main}.

\begin{thm}\label{thm ito formula}
  Let $\xi\in  L^2(\Omega,\sF_T,H)$, $q>1$, $q'=\frac{q}{q-1}$, $f\in M^{q'}(0,T;V')$ and
  $h\in M^2(0,T;L(U_2,H))$. Assume  $(u,v)\in M^q(0,T;V)\times M^2(0,T;L(U_1,H))$
   and that the
  following BDSDES:
  \begin{equation}\label{BDSES trivialcase}
    u(t)=\xi+\int_t^T f(s)\ ds+\int_t^T h(s)\ d\overleftarrow{B}(s)-\int_t^T v(s)\ dW_s,\ t\in
    [0,T]
  \end{equation}
  holds in the weak sense of Definition \ref{definition solution}.
  Then we assert that $u\in S^2(0,T;H)$ and the following It\^o formula
  holds almost surely
  \begin{equation}\label{Ito formula square}
    \begin{split}
        \|u(t)\|^2=\ &\|\xi\|^2+\int_t^T\left( 2\langle f(s),\ u(s)\rangle
                    +\|h(s)\|^2_2-\|v(s)\|_1^2\right)\ ds\\
                    &+\int_t^T 2 \langle u(s),\ h(s)\ d\overleftarrow{B}_s\rangle
                    -\int_t^T2\langle u(s),\ v(s)\ dW_s\rangle
    \end{split}
  \end{equation}
    for all $t\in[0,T]$.
\end{thm}
Here, we note that some techniques to prove Theorem \ref{thm ito formula} are borrowed from \cite{PrevotRockner2007,RenRocknerWang2007} and for the reader's convenience, we give the proof in the appendix.
\begin{lem}\label{lem uniqueness}
The solution in Theorem \ref{thm main} is unique.
\end{lem}
\begin{proof}
	Suppose $(u^1,v^1)$ and $(u^2,v^2)$ are two solutions of \eqref{bdsdes} in
	$$\left(M^{pq/2,q}(0,T;V)\cap S^p(0,T;H)\right)\times M^{p,2}(0,T;L(U_1,H)).$$
	Letting $(\bar{u},\bar{v})=(u^1-u^2,v^1-v^2)$, then by the product rule, It\^o formula and assumption $(\cA 2)$, we obtain
  \begin{equation}
    \begin{split}
      &E\left[e^{K_1t}\|\bar{u}(t)\|^2\right]\\
      =\ &
      E\bigg[
            \int_t^Te^{K_1s}\big(2\langle F(s,u^1(s),v^1(s))-F(s,u^2(s),v^2(s)),\ \bar{u}(s)\rangle  \\
        & \ \ \ \ \   +\|J(s,u^1(s),v^1(s))-J(s,u^2(s),v^2(s))\|_2^2
            -K_1\|\bar{u}(s)\|^2-\|\bar{v}(s)\|_1^2\big)\ ds\bigg]\\
      \leq\ &
      (\delta-1)E\left[ \int_t^Te^{K_1s}\|\bar{v}(s)\|_1^2\ ds  \right],\  t\in [0,T]
    \end{split}
  \end{equation}
  which implies $$ E\left[e^{K_1t}\|\bar{u}(t)\|^2+\int_t^Te^{K_1s}\|\bar{v}(s)\|_1^2\ ds\right]\leq 0.  $$
  Thus, $(\bar{u},\bar{v})=0$ $\mathbb{P}\otimes dt$-a.e..
  The path-wise uniqueness follows from the path continuity of $u^1,u^2$ in $H$. We complete the proof.
\end{proof}
\begin{rmk}\label{rmk unique}
  From the proof of Lemma \ref{lem uniqueness}, it follows that the uniqueness is only implied by assumptions $(\cA 2)$ and $(\cA 4)$.
\end{rmk}

In recent years, the monotonicity method (for instance, see \cite{Browder63,Minty62,Showalter96,NonlinearFuncAna}) is generalized and intensively used to analyze SPDEs (for example, see \cite{Krylov_Rozovskii81,LiuwRockner2010,Pardoux1979,PrevotRockner2007,RenRocknerWang2007,ZhangXich2009}) and BSPDEs (see \cite{MarquezDuranReal2004,QiuTangMPBSPDE11,ZhangXich2009}).
In the present paper, we shall generalize it to investigate the BDSDESs. Now, we show a useful lemma which plays an important role in the variational approach and will be used frequently below.
\begin{lem}\label{lem in pre variation}
    Let $p\geq 2$, $q>1$ and $\varsigma \in M^{p/2,1}(\tau,T)$ with $\tau$ ($0\leq\tau< T$) being one stopping time with respect to the backward filtration $\{\sF_{t,T}^B,t\in [0,T]\}$. The pair $(F,J)$ satisfy
     assumptions $(\cA 1),(\cA 2)$ and $(\cA 4)$ with $0\leq \beta\leq p-2$ on $[\tau,T]:=\{(\omega,t):t\in [\tau(\omega),T]\}$.  Moreover, we assume that there hold the following
    \begin{equation*}
      \begin{split}
        &\textrm{(a) }u^n\longrightarrow u\textrm{ weakly in }M^{pq/2,q}(\tau,T;V),\textrm{ as }n\rightarrow\infty ;\\
        &\textrm{(b) }u^n\longrightarrow u\textrm{ weakly star in }L^{p}(\Omega,L^{\infty}([\tau,T],H))\textrm{ as }n\rightarrow\infty ;\\
        &\textrm{(c) }v^n\longrightarrow v \textrm{ weakly in }M^{p,2}(\tau,T;L(U_1,H))\textrm{ as }n\rightarrow\infty ;\\
        &\textrm{(d) }G^n\longrightarrow G \textrm{ strongly in } L^p(\Omega,\sF_T,H) \textrm{ as }n\rightarrow\infty;   \\
        &\textrm{(e) }F^n(\cdot,u^n(\cdot),v^n(\cdot))\longrightarrow \bar{F} \textrm{ weakly in }M^{q'}(\tau,T;V')\textrm{ as }n\rightarrow\infty ;\\
        &\textrm{(f) }
	\textrm{for }\mathbb{P}\otimes dt\textrm{-almost } (\omega,t)\in \Omega\times[0,T],
        \lim_{n\rightarrow \infty}\|F^n(\omega,t,\varphi,\xi)-F(\omega,t,\varphi,\xi)\|_*=0\\
        &\textrm{ and }
        \lim_{n\rightarrow \infty}\|J^n(\omega,t,\varphi,\xi)-J(\omega,t,\varphi,\xi)\|_2=0
        \textrm{ hold for all }\varphi\in V \textrm{ and all } \xi\in L(U_1,H);       \\
        &\textrm{(g) }J^n(\cdot,u^n(\cdot),v^n(\cdot))\longrightarrow \bar{J} \textrm{ weakly in }M^{2}(\tau,T;L(U_2,H))\textrm{ as }n\rightarrow\infty ;\\
        &\textrm{(h) for each } n\in\bN,\\
        &\ \ u^n(t)=G^n+\int_t^T F^n(s,u^n(s),v^n(s))\ ds+\int_t^TJ^n(s,u^n(s),v^n(s))\ d\overleftarrow{B}_s\\
                    &~~~~~~~~~~~~~~-\int_t^Tv^n(s)\ dW_s,\textrm{ holds in the weak sense of Definitioin } \ref{definition solution},\\
      \end{split}
    \end{equation*}
    where for each $n\in\bN$, the pair $(F^n,J^n)$ satisfies assumptions $(\cA 2)$ and $(\cA 4)$  on $[\tau,T]$.

    Then
    $(u,v)\in\left(M^{pq/2,q}(\tau,T;V)\cap S^p(\tau,T;H)\right)\times M^{p,2}(0,T;L(U_1,H))$
	 is the unique solution
    to \eqref{bdsdes}.
\end{lem}

\begin{proof}
    Without any loss of generality, we take $\tau \equiv 0$.
     Define
     $$\bar{u}(t)=G+\int_t^T \bar{F}(s)\ ds+\int_t^T\bar{J}(s)\ d\overleftarrow{B}_s
                    -\int_t^Tv(s)\ dW_s,\ t\in [0,T].$$
		    From assertion (iv) of Lemma \ref{lem in prelimilary}, it follows that $u(\omega,t)=\bar{u}(\omega,t)$ for almost $\mathbb{P}\otimes dt$-$(\omega,t)\in [0,T]$. Identify $u$ with its modification $\bar{u}$. Then by Theorem \ref{thm ito formula}, we conclude that $u$ is an $H$-valued continuous process and thus $u\in S^p(0,T;H)$.
     It remains for us to prove $(F(\cdot,u(\cdot),v(\cdot)),J(\cdot,u(\cdot),v(\cdot)))
     =(\bar{F}(\cdot),\bar{J}(\cdot))\ \mathbb{P}\otimes dt$-a.e..

    For every
    $$(\varphi,\xi)\in\!\! \left(L^{pq/2}(\Omega,\sF;L^q(0,T;V))\!\cap\! L^p(\Omega,\sF; L^{\infty}(0,T;H))\right)
    \times L^p\left(\Omega,\sF;L^2(0,T;L(U_1,H))\right)$$
     it follows from (f)
    and the domination convergence theorem that
    \begin{equation}\label{proof of lem pre1}
      \begin{split}
        &\lim_{n\rightarrow \infty}E\bigg[\int_0^T\Big(\|J^n(s,\varphi(s),\xi(s))-J(s,\varphi(s),\xi(s))\|_2^2
        \\
        &+2\la J^n(s,u^n(s),v^n(s))-J^n(s,\varphi(s),\xi(s))
            ,\ J^n(s,\varphi(s),\xi(s))-J(s,\varphi(s),\xi(s))  \ra_2\\
      &+2\langle F^n(s,\varphi(s),\xi(s))-F(s,\varphi(s),\xi(s)),\ u^n(s)-\varphi(s)\rangle \Big)\ ds\bigg]=0.
      \end{split}
    \end{equation}

    On the other hand, we have
    \begin{align}
        &\langle F^n(s,u^n(s),v^n(s)),\ u^n(s)\rangle\nonumber\\
        =&
            \ \langle F^n(s,u^n(s),v^n(s))-F^n(s,\varphi(s),\xi(s)),\ u^n(s)-\varphi(s)\rangle \nonumber\\
        &
            \ +\langle F^n(s,\varphi(s),\xi(s))-F(s,\varphi(s),\xi(s)),\ u^n(s)-\varphi(s)\rangle
            \nonumber\\
        &
            \ +\langle F^n(s,u^n(s),v^n(s))-F(s,\varphi(s),\xi(s)),\ \varphi(s)\rangle
            +\langle F(s,\varphi(s),\xi(s)),\ u^n(s)\rangle ,\nonumber\\
        &
            \|J^n(s,u^n(s),v^n(s))\|_2^2\nonumber\\
        =&
            \ \|J^n(s,u^n(s),v^n(s))-J^n(s,\varphi(s),\xi(s))\|_2^2
            +\|J^n(s,\varphi(s),\xi(s))-J(s,\varphi(s),\xi(s))\|_2^2\nonumber\\
        &
            \ +2\la J^n(s,u^n(s),v^n(s))-J^n(s,\varphi(s),\xi(s))
            ,\ J^n(s,\varphi(s),\xi(s))-J(s,\varphi(s),\xi(s))  \ra_2\nonumber\\
        &
            \ +\langle J(s,\varphi(s),\xi(s)),\ J^n(s,u^n(s),v^n(s))-J(s,\varphi(s),\xi(s))\rangle_2
            \nonumber\\
        &\quad\ +\langle J^n(s,u^n(s),v^n(s)),\ J(s,\varphi(s),\xi(s))\rangle_2,\nonumber\\
            &\|v^n(s)\|_1^2
            =\|v^n(s)-\xi(s)\|_1^2+\la v^n(s)-\xi(s),\ \xi(s)  \ra_1
            +\la \xi(s),\ v^n(s)  \ra_1,\nonumber\\
        &\|u^n(s)\|=\|u^n(s)-\varphi(s)\|_1^2+\la u^n(s)-\varphi(s),\ \varphi(s)  \ra
            +\la \varphi(s),\ u^n(s)  \ra   \nonumber
    \end{align}
    and in view of $(\cA 2)$,
    \begin{align*}
        &2\langle F^n(s,u^n(s),v^n(s))-F^n(s,\varphi(s),\xi(s)),\ u^n(s)-\varphi(s)\rangle \\
        &
            +\|J^n(s,u^n(s),v^n(s))-J^n(s,\varphi(s),\xi(s))\|_2^2
            -\|v^n(s)-\xi(s)\|_1^2-K_1\|u^n(s)-\varphi(s)\|^2 \leq 0.
    \end{align*}
    Therefore, by Theorem \ref{thm ito formula} and the product rule, we have almost surely
    \begin{align*}
      &e^{K_1t}\|u^n(t)\|^2                             \\
      =\ &
            e^{K_1 T}\|G^n\|^2+\int_t^T2e^{K_1s}\langle u^n(s),\ J^n(s,u^n(s),v^n(s))\ d\overleftarrow{B}_s\rangle
        \\
      &
            +\int_t^Te^{K_1s}\Big[2\langle F^n(s,u^n(s),v^n(s)),\ u^n(s)\rangle -\|v^n(s)\|_1^2+\|J^n(s,u^n(s),v^n(s))\|_2^2
        \\
      &
            \quad\ -K_1\|u^n(s)\|^2\Big]\ ds
            -\int_t^T2e^{K_1s}\langle u^n(s),\ v^n(s)\ dW_s\rangle \\
      \leq\ &
            e^{K_1 T}\|G^n\|^2
            +\int_t^T2e^{K_1s}\langle u^n(s),\ J^n(s,u^n(s),v^n(s))\ d\overleftarrow{B}_s\rangle \\
        &
            +\int_t^Te^{K_1s}\Big[2\langle F^n(s,u^n(s),v^n(s))-F^n(s,\varphi(s),\xi(s)),\ u^n(s)-\varphi(s)\rangle \\
        &
            +\|J^n(s,u^n(s),v^n(s))-J^n(s,\varphi(s),\xi(s))\|_2^2
            -\|v^n(s)-\xi(s)\|_1^2-K_1\|u^n(s)-\varphi(s)\|^2\Big]\ ds\\
        &
            +\int_t^Te^{K_1s}\Big[ \|J^n(s,\varphi(s),\xi(s))-J(s,\varphi(s),\xi(s))\|_2^2
        \\
        &
            \quad\ +2\la J^n(s,u^n(s),v^n(s))-J^n(s,\varphi(s),\xi(s))
            ,\ J^n(s,\varphi(s),\xi(s))-J(s,\varphi(s),\xi(s))  \ra_2
        \\
        &
            \quad\ +2\langle F^n(s,\varphi(s),\xi(s))-F(s,\varphi(s),\xi(s)),\ u^n(s)-\varphi(s)\rangle \\
        &
            \quad\ +2\langle F^n(s,u^n(s),v^n(s))-F(s,\varphi(s),\xi(s)),\ \varphi(s)\rangle +2\langle F(s,\varphi(s),\xi(s)),\ u^n(s)\rangle \\
        &
            \quad\ +\langle J(s,\varphi(s),\xi(s)),\ J^n(s,u^n(s),v^n(s))-J(s,\varphi(s),\xi(s))\rangle_2
            \\
        &
            \quad\ +\langle J^n(s,u^n(s),v^n(s)),\ J(s,\varphi(s),\xi(s))\rangle_2 \\
        &
            \quad\ -\langle \xi(s),\ v^n(s)-\xi(s)\rangle_1 -\langle v^n(s),\ \xi(s)\rangle_1 -K_1\langle \varphi(s),\ u^n(s)-\varphi(s)\rangle
            \\
        &
            \quad\ -K_1\langle u^n(s),\ \varphi(s)\rangle  \Big]\ ds-\int_t^T2e^{K_1s}\langle u^n(s),\ v^n(s)\ dW_s\rangle \\
      \leq\  &
            e^{K_1 T}\|G^n\|^2+\int_t^T2e^{K_1s}\langle u^n(s),\ J^n(s,u^n(s),v^n(s))\ d\overleftarrow{B}_s\rangle
            \\
        &
            +\int_t^Te^{K_1s}\Big[\|J^n(s,\varphi(s),\xi(s))-J(s,\varphi(s),\xi(s))\|_2^2
            \\
        &
            \quad\ +2\la J^n(s,u^n(s),v^n(s))-J^n(s,\varphi(s),\xi(s))
            ,\ J^n(s,\varphi(s),\xi(s))-J(s,\varphi(s),\xi(s))  \ra_2
        \\
        &
            \quad\ +2\langle F^n(s,\varphi(s),\xi(s))-F(s,\varphi(s),\xi(s)),\ u^n(s)-\varphi(s)\rangle
            \\
        &
            \quad\ +2\langle F^n(s,u^n(s),v^n(s))-F(s,\varphi(s),\xi(s)),\ \varphi(s)\rangle
            +2\langle F(s,\varphi(s),\xi(s)),\ u^n(s)\rangle \\
        &
            \quad\ +\langle J(s,\varphi(s),\xi(s)),\ J^n(s,u^n(s),v^n(s))-J(s,\varphi(s),\xi(s))\rangle_2
            \\
        &
            \quad\ +\langle J^n(s,u^n(s),v^n(s)),\ J(s,\varphi(s),\xi(s))\rangle_2 \\
        &
            \quad\ -\langle \xi(s),\ v^n(s)-\xi(s)\rangle_1 -\langle v^n(s),\ \xi(s)\rangle_1
             -K_1\langle \varphi(s),\ u^n(s)-\varphi(s)\rangle
            \\
        &
            \quad\ -K_1\langle u^n(s),\varphi(s)\rangle  \Big]\ ds
            -\int_t^T2e^{K_1s}\langle u^n(s),\ v^n(s)\ dW_s\rangle ,\ \forall t\in [0,T].
    \end{align*}
Letting $n\rightarrow \infty$, by \eqref{proof of lem pre1} and the lower continuity of weak convergence, we obtain for every nonnegative $\psi\in L^{\infty}([0,T],\bR^+)$,
\begin{equation}\label{proof of lem pre2}
  \begin{split}
    &E\left[ \int_0^T \psi(t)\left(e^{K_1t}\|u(t)\|^2-e^{K_1T}\|G\|^2\right)\ dt  \right]\\
    \leq\ &\liminf_{n\rightarrow\infty}E\left[ \int_0^T
           \psi(t)\left(e^{K_1t}\|u^n(t)\|^2-e^{K_1T}\|G^n\|^2\right)\ dt  \right]\\
    \leq\ &E\bigg[\int_0^T
           \psi(t)\Big(\int_t^Te^{K_1s}\big(
           2\langle \bar{F}(s)-F(s,\varphi(s),\xi(s)),\ \varphi(s)\rangle
           \\
      &\ +2\langle F(s,\varphi(s),\ \xi(s)),\ u(s)\rangle +\langle J(s,\varphi(s),\xi(s)),\ \bar{J}(s)-J(s,\varphi(s),\xi(s))\rangle
      \\
      &\ +\langle \bar{J}(s),\ J(s,\varphi(s),\xi(s))\rangle -\langle \xi(s),\ v(s)-\xi(s)\rangle -\langle v(s),\ \xi(s)\rangle
      \\
      &\ -K_1\langle \varphi(s),\ u(s)-\varphi(s)\rangle -K_1\langle u(s),\ \varphi(s)\rangle \big)\ ds
                    \Big)\ dt\bigg].
  \end{split}
\end{equation}
As
\begin{equation}\label{proof of lem pre3}
  \begin{split}
    &E\left[ e^{K_1t}\|u(t)\|^2-e^{K_1T}\|G\|^2   \right]\\
    =&\
        E\Big[\int_t^Te^{K_1s}\left(2\langle \bar{F}(s),\ u(s)\rangle +\|\bar{J}(s)\|_2^2-\|v\|_1^2
        -K_1\|u(s)\|^2\right)\ ds   \Big],
  \end{split}
\end{equation}
by inserting \eqref{proof of lem pre3} into \eqref{proof of lem pre2} we obtain
\begin{equation}\label{proof of lem pre variation}
  \begin{split}
   E&\bigg[\int_0^T\psi(t)\Big(\int_t^T e^{K_1s}\big[2\langle \bar{F}(s)-F(s,\varphi(s),\xi(s)),\ u(s)-\varphi(s)\rangle -\|v(s)-\xi(s)\|_1^2
    \\
    &+\|\bar{J}(s)-J(s,\varphi(s),\xi(s))\|_2^2
    -K_1\|u(s)-\varphi(s)\|^2    \big]\ ds\Big)\ dt      \bigg]\leq 0.
  \end{split}
\end{equation}
Taking $(\varphi,\xi)=(u,v)$ we obtain $J(\cdot,u(\cdot),v(\cdot))=\bar{J}$. Finally, fist applying \eqref{proof of lem pre variation} to $(\varphi,\xi)=(u-\eps\bar{\phi}h,v)$ for $\eps>0,\bar{\phi}\in L^{\infty}(\Omega\times [0,T],\sF\otimes \cB([0,T]))$ and $h\in V$, then dividing both sides by $\eps$ and letting $\eps\downarrow 0$, by $(\cA 1)$, $(\cA 4)$ and  the dominated convergence theorem,   we obtain
\begin{equation}
  \begin{split}
   E\left[\int_0^T\psi(t)\left( \int_t^Te^{k_1s}\bar{\phi}(s)\langle \bar{F}(s)-F(s,u(s),v(s)),\ h\rangle \ ds  \right)   \ dt\right]\leq 0
  \end{split}
\end{equation}
which, together with the arbitrariness of $\psi,h$ and $\bar{\phi}$, implies $\bar{F}=F(\cdot,u,v)$.

Hence $(u,v)$ is a solution of \eqref{bdsdes} and the uniqueness follows from Lemma \ref{lem uniqueness} and Remark \ref{rmk unique}.
\end{proof}
\begin{rmk}\label{rmk lem in pre variation}
  In view of the proof of Lemma \ref{lem in pre variation}, we can replace the assumption (f) by the following one:
  \begin{align*}
    &~~~~~~~\textrm{ for } dt \textrm{-almost } t\in[0,T],\\
    &
        \lim_{n\rightarrow \infty}E[\|F^n(\omega,t,\varphi,\xi)-F(\omega,t,\varphi,\xi)\|_*^{q'}
        +\|J^n(\omega,t,\varphi,\xi)-J(\omega,t,\varphi,\xi)\|_2^2]=0,
          \\
        &~~~~~~~\textrm{holds for all }(\varphi,\xi)\in V\times L(U_1,H).
  \end{align*}
\end{rmk}

\begin{rmk}\label{rmk-lem-pre-variation}
  Indeed, instead of $(F^n,J^n)$ satisfying $(\mathcal {A}2)$, $(\mathcal {A}4)$ on $[\tau, T]$ and (f) (or Remark \ref{rmk lem in pre variation}), in order to obtain the assertion of Lemma \ref{lem in pre variation}, we need only to find $(\tilde{F}^n,\tilde{F}^n)$ satisfying $(\mathcal {A}2)$, $(\mathcal {A}4)$, (e), (g) and (f) (or Remark \ref{rmk lem in pre variation}) such that
  \begin{align*}
    &2\langle F^n(s,u^n(s),v^n(s)),\ u^n(s)\rangle +\|J^n(s,u^n(s),v^n(s))\|_2^2
    \\
    \leq&
    \,  2\langle \tilde{F}^n(s,u^n(s),v^n(s)),\ u^n(s)\rangle +\|\tilde{J}^n(s,u^n(s),v^n(s))\|_2^2,\ a.e.\,(\omega,s)\in\Omega\times[0,T].
  \end{align*}
   We can verify this claim in a similar way to the proof of Lemma \ref{lem in pre variation}.
\end{rmk}

\section{Proof of Theorem \ref{thm main}}
\subsection{The finite dimensional case}

\begin{thm}\label{thm finite dim}
  let $l,m,n\in\bN$ and $V=H=V'=\bR^n,U_1=\bR^m,U_2=\bR^l$. Then under the assumptions in
  Theorem \ref{thm main} there exists a unique solution pair
  $(u,v)\in (S^p(0,T;H)\cap M^{pq/2,q}(0,T;V))\times M^{p,2}(0,T;L(U_1,H))$ to BDSDES
  \eqref{bdsdes}.
\end{thm}
Before the proof of Theorem \ref{thm finite dim}, we show the following lemma which gives the estimates to the solution pair $(u,v)$ of BDSDES \eqref{bdsdes} in Theorem \ref{thm finite dim}.
\begin{lem}\label{rmk thm finite dim estimate}
	Under the assumptions in Theorem \ref{thm finite dim}, if $$(u,v)\in \left(S^p(0,T;H)\cap M^{pq/2,p}(0,T;V)\right)\times M^{p,2}(0,T;L(U_1,H))$$ is a solution to the equation \eqref{bdsdes-integral},
     there holds the following estimate
  \begin{equation}\label{estimate finite dim}
    \begin{split}
      &\|u\|_{S^{p}(0,T;H)}+\|u\|_{M^q(0,T;V)}
  +\|u\|^{q/2}_{M^{pq/2,q}(0,T;V))}+\|v\|_{M^{p,2}(0,T;L(U_1,H))}\\
  &\leq C \left\{\right \|G\|_{L^p(\Omega,\sF_T,H)}+\|\varsigma\|^{1/2}_{M^{p/2,1}(0,T)}  \},
    \end{split}
  \end{equation}
  where $C$ is a nonnegative constant depending on $T,K,q,p,\delta,\alpha$ and
  $\alpha_1$.
\end{lem}
\begin{proof}
  By It\^o formula, we have
    \begin{align*}
        &\|u(t)\|^p+\frac{p}{2}\int_t^T\!\! \|u(s)\|^{p-2}\|v(s)\|_1^2\ ds+\frac{p(p-2)}{2}\int_t^T
            \!\!\|u(s)\|^{p-4}\langle u(s) ,\ v(s)v^*(s)u(s)\rangle  ds\\
        =
        &\
            \|G\|^p+p\int_t^T\|u(s)\|^{p-2}\langle F(s,u(s),v(s)),\ u(s)\rangle \ ds\\
        &\
            +p\int_t^T\|u(s)\|^{p-2}\langle u(s),\ J(s,u(s),v(s))d\overleftarrow{B}_s\rangle
            -p\int_t^T\|u(s)\|^{p-2}\langle u(s),\ v(s)dW_s\rangle \\
        &\
            +\frac{p(p-2)}{2}\int_t^T \|u(s)\|^{p-4}\langle u(s),\ J(s,u(s),v(s))J(s,u(s),v(s))^*u(s)\rangle \ ds
            \\
        &\
            +\frac{p}{2}\int_t^T \|u(s)\|^{p-2}\|J(s,u(s),v(s))\|_2^2\ ds \\
      \leq
        &\
            \|G\|^p
            +p\int_t^T\!\!\!\|u(s)\|^{p-2}\!\langle u(s),\ J(s,u(s),v(s))d\overleftarrow{B}_s\rangle
            -p\int_t^T\!\!\!\|u(s)\|^{p-2}\!\langle u(s),\ v(s)dW_s\rangle \\
        &\
            +\frac{p(p-2)}{2}\int_t^T\!\!\Big[\|u(s)\|^{p-2}\left(
            K\|u(s)\|^2+K\|J(s,0,0)\|_2^2+\alpha_1\|u(s)\|_V^q\right)
                          +\|u\|^{p-4}\langle u(s) , \\
        &\
            \ \quad v(s)v(s)^*u(s)\rangle \! \Big]\, ds
            +\frac{p}{2}\int_t^T\|u(s)\|^{p-2}\left(
            K\|u(s)\|^2+\delta\|v(s)\|_1^2
            +\varsigma(s)-\alpha\|u(s)\|_V^q \right)\ ds\\
      \leq
        &\
            \|G\|^p+p\int_t^T\!\!\!\|u(s)\|^{p-2}\!\langle u(s),\ J(s,u(s),v(s))d\overleftarrow{B}_s\rangle
            -p\int_t^T\!\!\!\|u(s)\|^{p-2}\!\langle u(s),\ v(s)dW_s\rangle \\
        &\
            +\frac{p\delta}{2}\int_t^T\!|!\!\|u(s)\|^{p-2}\|v(s)\|_1^2\ ds
            -\frac{p(\alpha-(p-2)\alpha_1)}{2}\int_t^T\!\!\!\|u(s)\|^{p-2}\|u(s)\|_V^q
            \\
          &+C\left[\left(\int_0^T\varsigma(s)\ ds\right)^{p/2}+\int_t^T\|u(s)\|^p\ ds\right]+\frac{1}{4}\sup_{s\in[t,T]}\|u(s)\|^p\\
        &\
          +\frac{p(p-2)}{2}\int_t^T \|u(s)\|^{p-4}\langle u(s) ,\ v(s)v(s)^*u(s)\rangle \ ds,\ t\in[0,T]
    \end{align*}
    which together with the following
    \begin{align*}
        &
            E\bigg[\sup_{t\in [\tau,T]}\Big(\left| \int_t^T\!\!\!\|u(s)\|^{p-2}\!\langle u(s),
            \ J(s,u(s),v(s))d\overleftarrow{B}_s\rangle  \right|
        \\
        &
            \quad\ \ \ ~~~~~~~~+\left| \int_t^T\!\!\!\|u(s)\|^{p-2}\!\langle u(s),\  v(s)dW_s\rangle  \right|\Big)\bigg]\\
    \leq&\
            CE\bigg[\left| \int_{\tau}^T\!\!\|u(s)\|^{2p-2}(\varsigma(s)+\|u(s)\|^2+\|u(s)\|_V^q+\|v(s)\|_1^2)\ ds \right|^{1/2}\\
        &\
            +\left| \int_{\tau}^T\!\!\!\|u(s)\|^{2p-2}\|v(s)\|^2\ ds \right|^{1/2}\bigg] \\
    \leq &\
            \eps_1 E\left[ \sup_{t\in [\tau,T]}\|u(t)\|^p \right]
            +C(\eps_1,K,T)E\bigg[\int_{\tau}^T\!\!\!
            \|u(s)\|^p\ ds+\left(\int_{\tau}^T\!\!\!
            \varsigma(s)
            \ ds\right)^{p/2} \\
        &
            +\int_{\tau}^T
            \|u(s)\|^{p-2}\|v(s)\|_1^2\ ds
            +\int_{\tau}^T
            \|u(s)\|^{p-2}\|u(s)\|_V^q\ ds  \bigg],\ \tau\in[0,T]\\
    \end{align*}
    implies
      by Gronwall inequality and Young inequality  that
    \begin{equation}\label{eq 1 in lem finite esti}
          \begin{split}
        &E\left[\sup_{t\in [0,T]}\|u(t)\|^p\right]
        +E\left[\int_0^T\|u(s)\|^{p-2}(\|v(s)\|_1^2+\|u(s)\|_V^q)\ ds \right]\\
        \leq&\
        C \left\{E[\|G\|^p]+\left(\int_0^T\varsigma(s)\ ds\right)^{p/2}\right\}.
      \end{split}
    \end{equation}

     By It\^o formula, we have
    \begin{align*}
        &\|u(t)\|^2+\int_t^T\|v(s)\|_1^2\ ds\\
      =&\
            \|G\|^2+2\!\!\int_t^T\!\!\!\langle F(s,u(s),v(s)),\ u(s)\!\rangle \ ds
            +2\!\!\int_t^T\!\!\!\langle u(s),\ J(s,u(s),v(s))d\overleftarrow{B}_s \rangle \\
      &\
            -2\int_t^T\langle u(s),\ v(s)dW_s\rangle
            +\int_t^T \|J(s,u(s),v(s))\|_2^2\ ds\\
      \leq&\
            \|G\|^2+2\!\!\int_t^T\!\!\!\langle u(s),\ J(s,u(s),v(s))d\overleftarrow{B}_s \rangle
            -2\!\!\int_t^T\!\!\!\langle u(s),\ v(s)dW_s \rangle \\
      &\
            +K\!\!\int_t^T\!\!\!\!\|u(s)\|^2\ ds+\delta\!\!\int_t^T\!\!\!\!\|v(s)\|_1^2\ ds-\alpha\!\!\int_t^T\!\!\!\!\|u(s)\|_V^q\ ds
            +\int_0^T\!\!\!\!\varsigma(s)\ ds,~t\in[0,T].
    \end{align*}
    Taking $L^{p/2}(\Omega,\sF)$-norm on both sides and noticing that
    \begin{align*}
        &
            E\left[ \left|\int_{t}^T\!\!\!\! \langle u(s),\ v(s)dW_s \rangle \right|^{p/2}
            +\left|\int_{t}^T\!\!\!\! \langle u(s),\ J(s,u(s),v(s))d\overleftarrow{B}_s \rangle \right|^{p/2}  \right]
            \\
        \leq&\
            CE\left[  \left(\int_{t}^T\!\!\!\!\|u(s)\|^2\|v(s)\|_1^2\ ds\right)^{p/4}+
            \left(\int_{t}^T\!\!\!\!\|u(s)\|^2\|J(s,u(s),v(s))\|_1^2\ ds\right)^{p/4} \right]\\
        \leq&\
            \eps_2\left(\|v\|_{M^{p,2}(t,T;L(U_1,H)}^{p} +\|u\|_{M^{pq/2,q}(t,T;V)}^{pq/2}  \right)\\
        &\
            +C(\eps_2,p)\Big\{\|u\|_{S^p(0,T;H)}^p+\|\varsigma\|^{p/2}_{M^{p/2,1}(0,T)}
            \Big\},
    \end{align*}
    by the Young inequality and letting $\eps_2$ be small enough, we obtain
    \begin{equation}\label{proof of rmk finite estimate1}
      \begin{split}
      &\|v\|_{M^{p,2}(0,T;L(U_1,H))}+\|u\|^{q/2}_{M^{pq/2,q}(0,T;V)}          \\
      \leq &\ C \left\{\|u\|_{S^p(0,T;H)}
        +\|\varsigma\|^{1/2}_{M^{p/2,1}(0,T)}
       +\|G\|_{L^p(\Omega,\sF_T,H)}  \right\},
      \end{split}
    \end{equation}
which together with \eqref{eq 1 in lem finite esti} implies our estimate \eqref{estimate finite dim}. We complete the proof.
\end{proof}


To prove Theorem \ref{thm finite dim}, we need the flowing lemma which can be viewed as a
corollary of  \cite[Theorem 1.4]{PadouxPeng94BDSDE}. It is very likely
that this result has already appeared somewhere, but we have not seen it,
so we provide a proof here for the reader's convenience.

\begin{lem}\label{lem finite lipchitz case}
   Let $p\geq 2$, $H=V=V'=\bR^n$, $U_1=\bR^m$ and $U_2=\bR^l$, $n,m,l\in \bN$.
   Assume that $(f,h)$ satisfies $(\cA 2)$ and $(\cA 5)$, and that for all $(\omega,t)\in \Omega\times[0,T]$, $z\in \bR^{n\times m}$, $y_1,y_2\in \bR^n$,
  \begin{equation}\label{lipthitz wrt y lem finite-dim}
    \|f(\omega,t,y_1,z)-f(\omega,t,y_2,z)\|\leq K\|y_1-y_2\|,
  \end{equation}
  where the constant $K$ comes from assumption  $(\cA 5)$. Moreover, if $p>2$, we suppose $(\cA 6)$ holds for the pair $(f,h)$.
  Let $\xi\in L^p(\Omega,\sF_T,\bR^n)$
   and
  $$ E\left[ \left(\int_0^T\|f(t,0,0)\|\ dt\right)^p+\left(\int_0^T\|h(t,0,0)\|_2^2\ dt \right)^{p/2} \right]<\infty.  $$
  Then the backward doubly stochastic differential equation (BDSDE, for short)
  \begin{equation}\label{bdsde in lem finite dim}
    Y_t=\xi+\int_t^Tf(s,Y_s,Z_s)\ ds+\int_t^Th(s,Y_s,Z_s)\ d\overleftarrow{B}_s-\int_t^T Z_s\ dW_s,~~t\in[0,T]
  \end{equation}
  admits a unique solution $(Y,Z)\in S^p(0,T;\bR^n) \times M^{p,2}(0,T;\bR^{n\times m})$ such that
      \begin{equation}\label{estimate in lem finite dim}
      \begin{split}
        &E\left[\sup_{t\in [0,T]}\|Y(t)\|^p+\left(\int_0^T\|Z_s\|_1^2\ ds\right)^{p/2}\right]\\
        \leq&\
        C \left\{E[\|\xi\|^p]+\left(\int_0^T\|f(s,0,0)\|\ ds\right)^p+\left(\int_0^T \|h(s,0,0)\|_2^2 \ ds\right)^{p/2}  \right\}
      \end{split}
    \end{equation}
    where $C$ is a constant depending on $T,p,K,\delta,\alpha$ and $\alpha_1$.
\end{lem}

\begin{proof}
\tbf{Step 1}.
In a similar way to  the proof of Lemma \ref{rmk thm finite dim estimate}, we prove our estimate \eqref{estimate in lem finite dim}.
Indeed, the only difference lies in the fact that, by $(\cA 2)$ we have
\begin{equation*}
  \begin{split}
        \langle Y_s,\ f(s,Y_s,Z_s)\rangle +\|h(s,Y_s,Z_s)\|_2^2
    \leq
        K_1\|Y_s\|^2+\delta\|Z_s\|_1^2+\|Y_s\|\|f(s,0,0)\|+\|h(s,0,0)\|^2_2
  \end{split}
\end{equation*}
instead of the assumption $(\cA 3)$ on the pair $(f,h)$.

    \tbf{Step 2}. In a similar way to Lemma \ref{lem uniqueness}, we prove the uniqueness.
%

    \tbf{Step 3}. We prove the existence of the solution.

  Let
  \begin{equation*}
    \begin{split}
      &f^{N}(t,y,z)=f(t,y,z)-f(t,0,0)+f(t,0,0)1_{\{\|f(t,0,0)\|\leq N\}}, ~~ (t,y,z)\in [0,T]\times\bR^n\times\bR^{n\times m}.
    \end{split}
  \end{equation*}
Then by \cite[Theorem 1.4]{PadouxPeng94BDSDE}, there exists a unique solution $(Y^N,Z^N)\in S^p(0,T;\bR^n)\times M^{p,2}(0,T;\bR^{n\times m})$ to BDSDE \eqref{bdsde in lem finite dim} with $f$ replaced by $f^N$. Let $N,N'\in \bN$ and $N'>N$. Then through a similar procedure to \tbf{Step 1}., we obtain
\begin{equation}\label{proof of lem finite dim lem esti}
  \begin{split}
    &\sup_{N\in\bN}\left\{\|Y^N\|_{S^p(0,T;\bR^n)}+\|Y^N\|_{M^p(0,T;\bR^n)}+\|Z^N\|_{M^{p,2}(0,T;\bR^{n\times m})}\right\}\leq C,\\
    &\|Y^N-Y^{N'}\|_{S^2(0,T;\bR^n)}+\|Z^N-Z^{N'}\|_{M^{2}(0,T;\bR^{n\times m})}      \\
    \leq&\  C \|f(\cdot,0,0)1_{\{ \|f(\cdot,0,0)\in [N,N']\| \}}\|_{M^{2,1}(0,T;\bR^n)}
    \longrightarrow 0   \textrm{ as }N,N'\rightarrow \infty
  \end{split}
\end{equation}
with the constant $C$ independent of $N$.

Thus,
$(Y^N,Z^N)_{N\in \bN}$ is a Cauchy sequence in $S^2(0,T;\bR^n)\times M^{2}(0,T;\bR^{n\times m})$. Denote  the limit by $(Y,Z)$. From the Lipchitz continuity of $(f(\cdot,y,z),g(\cdot,y,z))$ with respect to $(y,z)$, it follows that
\begin{equation*}
  \begin{split}
        \|f^N(\cdot,Y^N_{\cdot},Z^N_{\cdot})
        -f(\cdot,Y_{\cdot},Z_{\cdot})\|_{M^{2,1}(0,T;\bR^n)}
        +\|h(\cdot,Y^N_{\cdot},Z^N_{\cdot})
        -h(\cdot,Y_{\cdot},Z_{\cdot})\|_{M^{2}(0,T;\bR^{n\times l})} \rightarrow 0.
  \end{split}
\end{equation*}
Hence, $(Y,Z)$ is a solution of BDSDE \eqref{bdsde in lem finite dim} in $S^2(0,T;\bR^n)\times M^{2}(0,T;\bR^{n\times m})$.

 On the other hand, in view of the first equation in \eqref{proof of lem finite dim lem esti}, we have $$(Y,Z)\in \left(L^p(\Omega,L^{\infty}([0,T],\bR^n))\cap S^2(0,T;\bR^n)\right)\times M^{p,2}(0,T;\bR^{n\times m}),$$
 which implies that $(Y,Z)$ lies in $S^p(0,T;\bR^n)\times M^{p,2}(0,T;\bR^{n\times m})$.
We complete the proof.
\end{proof}

\begin{proof}[Proof of Theorem \ref{thm finite dim}]

  \tbf{Step 1}.
  First of all, let us reduce assumption $(\cA 2)$ to the case of $K_1=0$.
  Assume $(u,v)$ is a solution to Equation \eqref{bdsdes-integral} and set
  \begin{equation*}
    \begin{split}
        &\theta(t):=e^{\frac{tK_1}{2}},                                 \\
        &\bar{u}(\omega,t):=\theta(t)u(\omega,t),                        \\
        &\bar{v}(\omega,t):=\theta(t)v(\omega,t),                        \\
        &\bar{J}(\omega,t,\bar{u}(t),\bar{v}(t)):=
        \theta(t)J(\omega,t,\theta(t)^{-1}\bar{u}(t),\theta(t)^{-1}\bar{v}(t)),                                           \\
        &\bar{F}(\omega,t,\bar{u}(t),\bar{v}(t)):=
        \theta(t)F(\omega,t,\theta(t)^{-1}\bar{u}(t),\theta(t)^{-1}\bar{v}(t))
        -\frac{1}{2}K_1\bar{u}(t).
    \end{split}
  \end{equation*}
    Then through  careful computations, we check that the pair
    $(\bar{F},\bar{J})$ also satisfies the same assumptions given to the pair $(F,J)$ only with the
    constant $K$ ($K_1$, respectively)
     replaced by another nonnegative constant $\bar{K}$ (0, respectively).
     Hence, we may assume $K_1=0$ in the following proof.

     \tbf{Step 2}.
     Take $r\geq p\vee(q+\beta)\vee \frac{pq}{2}$.  Assume further that $\|F(\cdot,0,0)\|_*\in L^r(\Omega,L^1([0,T],\bR))$
     , $\|J(\cdot,0,0)\|_2\in M^{r,2}(0,T;\bR) $ and $G\in L^r(\Omega,\sF_T,\bR^n)$.
    Fix $\tilde{v}\in M^{r,2}(0,T;\bR^n\times\bR^m)$.
     Consider the
    following backward doubly stochastic differential equation:
    \begin{equation}\label{proof of lem finite dim bdsde1}
      \begin{split}
        u(t)=G+\int_t^TF(s,u(s),\tilde{v}(s))\ ds+\int_t^TJ(s,u(s),v(s))\ d\overleftarrow{B}_s-\int_t^Tv(s)\ dW_s.
      \end{split}
    \end{equation}
    In this case, for every $(\omega,t)\in \Omega\times[0,T]$,
     $x\mapsto F(\omega,t,x,\tilde{v}(\omega,t))$ is a continuous monotone
     function on $\bR^n$. Let $F_{\eps}(\omega,t,\cdot)$ ($\eps>0$)  be the
     Yosida approximation of $F(\omega,t,\cdot,\tilde{v}(\omega,t))$, i.e.
    \begin{equation}
      \begin{split}
        &F_{\eps}(\omega,t,x):=\eps^{-1}(H_{\eps}(\omega,t,x)-x)
        =F(\omega,t,H_{\eps}(\omega,t,x))                \\
        &H_{\eps}(\omega,t,x):=(I-\eps
        F(\omega,t,\cdot,\tilde{v}(\omega,t)))^{-1}(x).
      \end{split}
    \end{equation}
    Then we conclude (c.f. \cite{NonlinearFuncAna}) that
     $x\mapsto H_{\eps}(\omega,t,x)$ is a homeomorphism on $\bR^n$ for
    each $(\omega,t)$ and that for any $x,y\in\bR^n$
    \begin{equation*}
    \begin{split}
      &\textrm{(a) }\|F_{\eps}(\omega,t,x)-F_{\eps}(\omega,t,y)\|\leq
        2\eps^{-1}\|x-y\|;                                                \\
      &\textrm{(b) }\|F_{\eps}(\omega,t,x)\|\leq
        \|F(\omega,t,x,\tilde{v}(\omega,t))\|;                            \\
      &\textrm{(c) }\langle F_{\eps}(\omega,t,x)-F_{\eps}(\omega,t,y),\ x-y\rangle \ \leq 0;\\
      &\textrm{(d) }\lim_{\eps\downarrow 0}
      \|F_{\eps}(\omega,t,x)- F(\omega,t,x,\tilde{v}(\omega,t))\|=0.
    \end{split}
    \end{equation*}
    It follows from (b), (c) and $(\cA 5)$ that for any $x\in \bR^n$
    \begin{equation}\label{proof of thm finite eps-coercivity}
        \begin{split}
      \langle  F_{\eps}(\omega,t,x),\ x\rangle \ \leq \|x\|\cdot\|F_{\eps}(\omega,t,0)\|
      \leq \left(\|F(\omega,t,0,0)\|+K\|\tilde{v}(\omega,t)\|_1\right)\cdot \|x\|.
        \end{split}
    \end{equation}

    By Lemma \ref{lem finite lipchitz case}, there exists a unique solution $(u^{\eps},v^{\eps})\in S^r(0,T;\bR^n)\times M^{r,2}(0,T;\bR^{n\times m})$ to the the following BDSDE:
    \begin{equation}\label{proof of thm finite case eps-eq}
      \begin{split}
        u^{\eps}(t)=G+\int_t^T\!\!\!\!F_{\eps}(s,u^{\eps}(s))\ ds
        +\int_t^T\!\!\!\!J(s,u^{\eps}(s),v^{\eps}(s))\ d\overleftarrow{B}_s-\int_t^T\!\!\!\!v^{\eps}(s)\ dW_s.
      \end{split}
    \end{equation}
    Since $V=H=V'=\bR^n$, in view of $(\cA 6)$ we have
    $$J(t,\varphi,\phi)J^*(t,\varphi,\phi)\leq \phi\phi^*+C(K,\alpha_1)\left(
    \|J(t,0,0)\|_2^2+\|\varphi\|^2
    \right)I,\ \forall \,(\varphi,\phi)\in\bR^n\times\bR^{n\times m}.  $$
  Thus, by using \eqref{proof of thm finite eps-coercivity} and in a similar way to the proof of Lemma \ref{rmk thm finite dim estimate},
   we obtain
  \begin{equation}\label{proof of thm finite estimate eps-case}
    \begin{split}
      &\|u^{\eps}\|_{S^r(0,T;\bR^n)}+\|v^{\eps}\|_{M^{r,2}(0,T;\bR^{n\times m})}
      \\
      \leq&\
            C\big\{ \|\xi\|_{L^r(\Omega,\sF_T,\bR^n)}+\|F(\cdot,0,0)\|_{M^{r,1}(0,T;\bR^n)}
            +\|J(\cdot,0,0)\|_{M^{r,2}(0,T;\bR^{n\times l})}\\
      &\
            +\|\tilde{v}\|_{M^{r,1}(0,T;\bR^{n\times m})}    \big\}\\
      \leq&\
            C\big\{\|\xi\|_{L^r(\Omega,\sF_T,\bR^n)}+\|F(\cdot,0,0)\|_{M^{r,1}(0,T;\bR^n)}
            +\|J(\cdot,0,0)\|_{M^{r,2}(0,T;\bR^{n\times l})}\\
      &\
            +T^{1/2}\|\tilde{v}\|_{M^{r,2}(0,T;\bR^{n\times m})}
            \big\}
    \end{split}
  \end{equation}
  with the constant $C$ independent of $\eps$.

  On the other hand, we have
  \begin{equation}\label{proof of thm finite estimate1 eps-case}
    \begin{split}
      &E\left[\int_0^T\|F_{\eps}(t,u^{\eps}(t))\|^{q'}\ dt  \right]\\
      \leq &\ E\left[\int_0^T\|F(t,u^{\eps}(t),\tilde{v}(t))\|^{q'}\ dt  \right]\\
      \leq &\ E\left[\int_0^T\left(\varsigma(t)+K(\|u^{\eps}(t)\|_V^q+\|u^{\eps}(t)\|^2+\|\tilde{v}(t)\|_1^2)\right)
      \left(1+\|u^{\eps}(t)\|^{\beta}\right)\ dt   \right]\\
 \leq&\  C(K,T,\|\varsigma\|_{M^{p/2,1}(0,T)},\|\tilde{v}\|_{M^{r,2}(o,T;\bR^{n\times m})},\|u^{\eps}\|_{S^r(0,T;\bR^n)})
 \end{split}
  \end{equation}
  and
\begin{equation}\label{proof of thm finite estimate2 eps-case}
  \begin{split}
    & E\left[\int_0^T\|J(t,u^{\eps}(t),v^{\eps}(t))\|_2^2dt\right]\\
     \leq&\
        CE\left[\int_0^T(\|J(t,0,0)\|_2^2+\|u^{\eps}(t)\|^2
        +\|v^{\eps}(t)\|_1^2)dt\right]\\
     \leq&\
        C\left\{\|J(\cdot,0,0)\|^2_{M^{r,2}(0,T;\bR^{n\times l})}+\|v^{\eps}\|_{M^{r,2}(0,T;\bR^{n\times m})}^2+\|u^{\eps}\|^2_{S^r(0,T;\bR^n)} \right\}
  \end{split}
\end{equation}
Combining \eqref{proof of thm finite estimate eps-case},
 \eqref{proof of thm finite estimate1 eps-case} and \eqref{proof of thm finite estimate2 eps-case}, we conclude that there exists a sequence
 $\eps_k\downarrow 0$ and $(\bar{u},\bar{v},\bar{F},\bar{J})$ such that
\begin{equation}
  \begin{split}
    u^{\eps_k}&\longrightarrow\bar{u} \textrm{ weakly star in } L^{r}(\Omega,L^{\infty}(0,T;\bR^n));\\
    u^{\eps_k}&\longrightarrow\bar{u} \textrm{ weakly in } M^r(0,T;\bR^n);\\
    v^{\eps_k}&\longrightarrow\bar{v} \textrm{ weakly in } M^{r,2}(0,T;\bR^{n\times m});\\
    F_{\eps_k}(\cdot,u^{\eps_k}(\cdot))&\longrightarrow\bar{F} \textrm{ weakly in } M^{q'}(0,T;\bR^n);\\
    J(\cdot,u^{\eps_k}(\cdot),v^{\eps_k}(\cdot))&\longrightarrow\bar{J} \textrm{ weakly in } M^{2}(0,T;\bR^{n\times l})
  \end{split}
\end{equation}
as $k\rightarrow\infty$.

By Lemma \ref{lem in pre variation}, $(\bar{u},\bar{v})\in S^r(0,T;\bR^n)\times M^{r,2}(0,T;\bR^{n\times m})$
is the unique solution to \eqref{proof of lem finite dim bdsde1}.

Take $v^0=0$. We consider the following Picard iteration: for $k\in\bN$, let $(u^k,v^k)\in M^r(0,T;\bR^n)\times M^{r,2}(0,T;\bR^{n\times m})$ be the
unique solution of \eqref{proof of lem finite dim bdsde1} with $\tilde{v}$ replaced by $v^{k-1}$ there. Set $(X^k,Z^k):=(u^{k+1}-u^k,v^{k+1}-v^k)$. In view of the assumption $(\cA 2)$ (with $K_1=0$), we have
\begin{align}
    &E\left[\|X^k(t)\|^2+\int_t^T\|Z^k(s)\|_1^2\ ds   \right]\nonumber\\
    =&\ E\bigg[
        2\int_t^T\!\!\langle X^k(s),\ F(s,u^{k+1}(s),v^k(s))-F(s,u^k(s),v^{k-1}(s)\rangle \ ds\nonumber\\
     &\   +\int_t^T\!\!\!\|J(s,u^{k+1}(s),v^{k+1}(s))-J(s,u^k(s),v^k(s))\|_2^2\ ds
        \bigg]
        \nonumber\\
    =&\
        E\bigg[
        \int_t^T\!\!\!\big\{
            2\langle X^k(s),\ F(s,u^{k+1}(s),v^{k+1}(s))-F(s,u^k(s),v^{k}(s))\rangle \nonumber\\
    &\
        +\|J(s,u^{k+1}(s),v^{k+1}(s))-J(s,u^k(s),v^k(s))\|_2^2\big\}\ ds\nonumber\\
    &\
        +2\int_t^T\!\!\!\big\{
        \langle X^k(s),\ F(s,u^{k+1}(s),v^{k}(s))-F(s,u^{k+1}(s),v^{k+1}(s))\rangle \nonumber\\
    &\
        +\langle X^k(s),\ F(s,u^k(s),v^k(s))-F(s,u^{k},v^{k-1}(s))\rangle \big\}\ ds
        \bigg]\nonumber\\
    \leq
    &\
        E\bigg[
        \int_t^T\!\!\!\delta\|Z^k(s)\|_1^2\ ds
        +c_0\int_t^T\!\!\!\|X^k(s)\|^2\ ds+\frac{1-\delta}{2}\int_t^T\!\!\!\|Z^k(s)\|_1^2\ ds\nonumber\\
    &+\frac{1-\delta}{4}\int_t^T\!\!\!\|Z^{k-1}(s)\|_1^2\ ds
    \bigg]    \label{proof of lem finite dim bdsde11}
\end{align}
Thus, for $\mu :=2c_0/(1-\delta)$

\begin{align*}
    &-\frac{d}{dt}\left( e^{\mu  t}\int_t^T\!\!\! E[\|X^k(s)\|^2]\ ds \right)
        +e^{\mu  t}\int_t^T\!\!\!E[\|Z^k(s)\|_1^2]\ ds\\
    \leq&\
        \frac{e^{\mu  t}}{2}\int_t^T\!\!\! E[\|Z^{k-1}(s)\|_1^2]\ ds=:\frac{a_k(t)}{2}.
\end{align*}
Integrating both sides from $0$ to $T$, we get
$$ \int_0^T\!\!\!E[\|X^k(s)\|^2]\ ds+\int_0^T\!\!\!a_{k+1}(t)\ dt\leq \frac{1}{2}\int_0^T\!\!\!a_k(t)\ dt\leq\frac{1}{2^k}\int_0^T\!\!\!a_1(t)\ dt=:\frac{c_1}{2^k}.  $$
Then it follows from \eqref{proof of lem finite dim bdsde11} that
$$\int_0^T\!\!\!E[\|Z^k(s)\|_1^2]\ ds\leq \frac{c_1\mu }{2^k}
+\frac{1}{2}\int_0^T\!\!\!E[\|Z^{k-1}(s)\|_1^2]\ ds$$
which yields
$$\int_0^T\!\!\!E[\|Z^k(s)\|_1^2]ds\leq \frac{(k\mu+1)c_1}{2^k}.$$
Therefore, there exits a pair $(u,v)\in M^2(0,T;\bR^n)\times M^2(0,T;\bR^{n\times m})$ such that
\begin{equation}\label{eq_lim}
\lim_{k\rightarrow\infty}\|u-u^k\|_{M^2(0,T;\bR^n)}+\|v^k-v\|_{M^2(0,T;\bR^{n\times m})}=0.
\end{equation}
From \eqref{proof of lem finite dim bdsde11} and the above estimates we also have
\begin{equation}\label{proof of lem finite dim bdsde12}
    \sup_{k\in \bN}\sup_{t\in [0,T]} E[\|u^k(t)\|^2]<+\infty.
\end{equation}
On the other hand, similar to \eqref{proof of thm finite estimate eps-case} we have
\begin{equation}
  \begin{split}
    b_k&(t)\\
    :=&\ \|u^{k}\|_{S^r(t,T;\bR^n)}+\|v^k\|_{M^{r,2}(t,T;\bR^{n\times m})}\\
    \leq&\
        C\left\{\|G\|_{L^r(\Omega,\sF_T,\bR^n)}+ \|F(\cdot,0,v^{k-1})\|_{M^{r,1}(t,T;\bR^n)}
        +\|J(\cdot,0,0)\|_{M^{r,2}(t,T;\bR^{n\times m})}   \right\}\\
    \leq&\
        C_0\{\|G\|_{L^r(\Omega,\sF_T,\bR^n)}+ \|F(\cdot,0,\cdot)\|_{M^{r,1}(t,T;\bR^n)}
        +\|J(\cdot,0,0)\|_{M^{r,2}(t,T;\bR^{n\times m})}\\
    &\
        +(T-t)^{1/2}\|v^{k-1}\|_{M^{r,2}(t,T;\bR^{n\times m})}   \}\\
    \leq&\
        C_1+C_1(T-t)^{1/2}\|v^{k-1}\|_{M^{r,2}(t,T;\bR^{n\times m})},\ t\in [0,T)
  \end{split}
\end{equation}
where the constants $C,C_0$ and $C_1$ are all independent of $k$ and $T-t$.
Choose $\tau=T-\frac{1}{ 4|C_1|^2}\wedge T$.
Then for any $t\in [\tau,T]$, we have
$$
b_k(t)\leq C_1+\frac{1}{2}b_{k-1}(t),\ b_1(t) < \infty,
$$
which implies $\sup_{k\in \bN}\{ \|u^{k}\|_{S^r(t,T;\bR^n)}+\|v^k\|_{M^{r,2}(t,T;\bR^{n\times m})}  \}
 \leq 2C_1+b_1(t)<\infty$.

Hence we have $(u,v)\in L^{r}(\Omega,L^{\infty}([\tau,T],\bR^n))\times M^{r,2}(\tau,T;\bR^{n\times m})$. By induction, we conclude that
\begin{equation}\label{proof of lem finite dim bdsde13}
(u,v)\in L^{r}(\Omega,L^{\infty}([0,T],\bR^n))\times M^{r,2}(0,T;\bR^{n\times m}).
\end{equation}

We now show that $(\bar{u},\bar{v})$ admits a version which is a solution to Eq. \eqref{bdsdes-integral}. In fact, let $(\bar{u},\bar{v})$ solve the following equation:
$$ \bar{u}(t)=G+\int_t^T\!\!\! F(s,\bar{u}(s),v(s))\ ds+\int_t^T\!\!\!J(s,\bar{u}(s),\bar{v}(s))\ d\overleftarrow{B}_s-\int_t^T\!\!\!\bar{v}(s)\ dW_s.
$$
Similar to \eqref{proof of lem finite dim bdsde11}, it follows that
\begin{equation*}
  \begin{split}
    &E\left[\|\bar{u}(t)-u^k(t)\|^2\right]+(1-\delta)\int_t^TE\left[\|\bar{v}(s)-v^k(s)\|_1^2\right]\ ds\\
    \leq&\
        c_0\int_t^TE\left[\|\bar{u}(s)-u^k(s)\|^2\right]\ ds
        +\frac{1-\delta}{2}\int_t^TE\left[\|\bar{v}(s)-v(s)\|_1^2\right]\ ds\\
        &\ +\frac{1-\delta}{4}\int_t^TE\left[\|v^k(s)-v^{k-1}(s)\|_1^2\right]\ ds.
  \end{split}
\end{equation*}
In view of \eqref{eq_lim}, we have
\begin{align*}
  &\lim_{k\rightarrow \infty} \int_t^TE\left[\|v^k(s)-v^{k-1}(s)\|_1^2\right]\ ds
  =0\quad\textrm{and}\\
  &\lim_{k\rightarrow \infty} \int_t^TE\left[\|\bar{v}(s)-v^k(s)\|_1^2\right]\ ds
  =
  \int_t^TE\left[\|\bar{v}(s)-v(s)\|_1^2\right]\ ds.
\end{align*}
Taking $\eta(t):=\limsup_{k\rightarrow\infty}E[\|u^k(t)-\bar{u}(t)\|^2]$, by \eqref{proof of lem finite dim bdsde12} and Fatou's lemma, we obtain
\begin{align*}
  \eta(t) &\leq c_0\int_t^T\!\eta(s)\ ds
  -\lim_{k\rightarrow \infty}\int_t^TE\left[(1-\delta)\|\bar{v}(s)-v^k(s)\|_1^2
   -\frac{1-\delta}{2}  \|\bar{v}(s)-v(s)\|_1^2  \right]\ ds\\
   &+\frac{1-\delta}{4}\lim_{k\rightarrow \infty} \int_t^TE\left[\|v^k(s)-v^{k-1}(s)\|_1^2\right]\ ds
   \leq c_0\int_t^T\!\eta(s)\ ds
\end{align*}
which implies that $\eta\equiv 0$ by Gronwall's inequality. Furthermore, we conclude
 $$\lim_{k\rightarrow\infty}\int_t^TE[\|\bar{v}(s)-v^k(s)\|_1^2]\ ds=
 \int_t^TE[\|\bar{v}(s)-v(s)\|_1^2]\ ds=0.$$
It follows that $(\bar{u},\bar{v})$ is a modification of $(u,v)$. By Theorem \ref{thm ito formula}, Lemma \ref{lem uniqueness} and estimate \eqref{proof of lem finite dim bdsde13}, we conclude that
$(\bar{u},\bar{v})\in S^r(0,T;\bR^n)\times M^{r,2}(0,T;\bR^{n\times m})$ is the unique solution to BDSDES \eqref{bdsdes-integral}.

 \tbf{Step 3}. For any $N>0$, denote
 $$ F^N:=F-F(\cdot,0,0)1_{\{ \|F(\cdot,0,0)\|_{*}\geq N \}},J^N:=J 1_{\{ \|J(\cdot,0,0)\|_{2}\leq N \}},G^N:=G 1_{\{\|G\|\leq N\}}.$$
 Then in view of \tbf{Step 2}, there exists a unique solution $(u^N,v^N)\in (S^p(0,T;H)\cap M^{pq/2,q}(0,T;V))\times M^{p,2}(0,T;L(U_1,H))$ to the following BDSDES:
 \begin{equation}
   \begin{split}
     u^N(t)=&\ G^N+\int_t^T\!\!\!\!F^N(s,u^N(s),v^N(s))\ ds+\int_t^T\!\!\!\!J^N(s,u^N(s),v^N(s))\ d\overleftarrow{B}_s\\
     &-\int_t^T\!\!\!\!v^N(s)\ dW_s.
   \end{split}
 \end{equation}

 Since
 $$
 \langle \varphi,\ F(t,0,0)1_{\{ \|F(\cdot,0,0)\|_{*}\geq N \}} \rangle \ \leq\  \epsilon \|\varphi\|_V^q+C(\epsilon) \|F(t,0,0)\|_*^{q'},\ \forall \varphi\in V,
  $$
for any $(\varphi,\phi)\in V\times L(U_1,H)$,  we have
\begin{equation*}
  \begin{split}
    \langle \varphi, F^N(t,\varphi,\phi) \rangle +\|J^N(t,\varphi,\phi)\|_2^2
    \leq
    -(\alpha-\epsilon) \|\varphi\|_V^q+\delta\|\phi\|_1^2+ C(\epsilon)
    \|F(t,0,0)\|_*^{q'}+\varsigma(t)
  \end{split}
\end{equation*}
with the positive constant $\epsilon$ waiting to be determined later.
Then choosing $\epsilon$ to be so small that $ \alpha-(p-2)\alpha_1-\epsilon>0$,
we check that the pair $(F^N,J^N)$ satisfies all the assumptions given to the pair $(F,J)$
with $\varsigma$ replaced by $C(\epsilon)
    \|F(\cdot,0,0)\|_*^{q'}+\varsigma\leq \left(C(\epsilon)+1\right)\varsigma$.
By Lemma \ref{rmk thm finite dim estimate}, we have
\begin{equation}
  \begin{split}
    &\|u^N\|_{S^p(0,T;H)}+\|u^N\|^{q/2}_{M^{pq/2,q}(0,T;V)}+\|v^N\|_{M^{p,2}(0,T;L(U_1,H))}  \\
    \leq&\
     C\left\{ \|\varsigma\|_{M^{p/2,1}(0,T)}^{1/2}
    +\|G^N\|_{L^p(\Omega,\sF_T,H)}
      \right\}\\
    \leq&\
     C\left\{\|\varsigma\|_{M^{p/2,1}(0,T)}^{1/2}
     +\|G\|_{L^p(\Omega,\sF_T,H)}
           \right\}.
  \end{split}
\end{equation}
On the other hand, in view of $(\cA 4)$, we have
\begin{equation}
  \begin{split}
        &E\left[\int_0^T\|F^N(t,u^{N}(t),v^N(t))\|_{*}^{q'}\ dt  \right]\\
    \leq &\ C
        E\left[\int_0^T\!\!\!\!\left(\|F(t,u^N(t),v^N(t))\|_*^{q'}+\|F(t,0,0)\|^{q'}_*\right)\ dt  \right]\\
    \leq &\
        C\bigg\{
        E\left[\int_0^T\!\!\!\!\left(\varsigma(t)+K(\|u^{N}(t)\|_V^q+\|u^N(t)\|^2+\|v^N(t)\|_1^2)\right)
        \left(1+\|u^{N}(t)\|^{\beta}\right)\ dt   \right]
            \bigg\}\\
    \leq&\
        C
        \Big\{\|\varsigma\|_{M^{p/2,1}(0,T)}^{p/2}+\|v^N\|_{M^{p,2}(0,T;L(U_1,H))}^p
        +\|u^{N}\|^p_{S^p(0,T;H)}+\|u^N\|^{pq/2}_{M^{pq/2,q}(0,T;V)}
                \Big\}\\
  \end{split}
\end{equation}
  and
\begin{equation}
  \begin{split}
    & E\left[\int_0^T\|J^N(t,u^{N}(t),v^{N}(t))\|_2^2\ dt\right]\\
     \leq&\
        CE\left[\int_0^T\left(\varsigma(t)+\|u^{N}(t)\|^2+\|u^N(t)\|_V^q
        +\|v^{N}(t)\|_1^2\right)\ dt\right]\\
     \leq&\
        C\left\{\|\varsigma\|_{M^{p/2,1}(0,T)}+\|v^{N}\|_{M^{p,2}(0,T;L(U_1,H))}^2
        +\|u^{N}\|^2_{S^p(0,T;H)}+\|u^N\|_{M^{pq/2,q}(0,T;V)}^{q} \right\}.
  \end{split}
\end{equation}
Thus,  there exists a subsequence $N_k$ and $u,v,\bar{F},\bar{J}$ such that
\begin{equation}
  \begin{split}
    u^{N_k}&\longrightarrow u \textrm{ weakly star in } L^{p}(\Omega,L^{\infty}(0,T;H));\\
    u^{N_k}&\longrightarrow u \textrm{ weakly in } M^{pq/2,q}(0,T;V);\\
    v^{N_k}&\longrightarrow v \textrm{ weakly in } M^{p,2}(0,T;L(U_1,H));\\
    F^{N_k}(\cdot,u^{\eps_k}(\cdot))&\longrightarrow \bar{F} \textrm{ weakly in } M^{q'}(0,T;V');\\
    J(\cdot,u^{N_k}(\cdot),v^{N_k}(\cdot))&\longrightarrow \bar{J} \textrm{ weakly in } M^{2}(0,T;L(U_2,H))
  \end{split}
\end{equation}
as $k\rightarrow\infty$. It is clear that, for any $(\varphi,\phi)\in V\times L(U_1,H)$
$$\lim_{k\rightarrow\infty}\|F^{N_k}(\cdot,0,0)1_{\{\|F(\cdot,0,0)\|_*\geq N_k\}}\|_*+\|J^{N_k}(\cdot,\varphi,\phi)1_{\{\|J(\cdot,0,0)\|_2\geq N_k\}}\|_2=0,\ \mathbb{P}\otimes dt\textrm{-a.e.}.$$
  By Lemma \ref{lem in pre variation},
  $$(u,v)\in \left( S^p(0,T;H)\cap M^{pq/2,q}(0,T;V) \right)\times M^{p,2}(0,T;L(U_1,H))$$
    is the unique solution to \eqref{bdsdes} and we complete our proof.
\end{proof}

\subsection{Proof of  Theorem \ref{thm main}}

Let $\{e_i|\ i\in \mathbb{N} \}\subset V$ be an orthonormal basis of H and
let $H_n:=\textrm{span}\{e_1,\dots,e_n\}$ such that span$\{e_i|i\in
\mathbb{N}\}$ is dense in  $V$. Let $P_n:V'\rightarrow H_n$ be defined by
$$ P_n \phi:=\sum_{i=1}^n\langle \phi,\ e_i\rangle e_i,\quad \phi\in V'. $$
Obviously, $P_n|H$ is just the orthogonal projection onto $H_n$.
Let $\{g^i_{1},g^i_2,\dots\}$ be an orthogonal basis of $U_i$, $i=1,2$ and
$$W^n(t):=P^1_nW_t:=\sum_{i=1}^n\langle W_t,\ g^1_i\rangle _{U_1}g^1_i,
\ B^n(t):=P^2_nB_t:=\sum_{i=1}^n\langle B_t,\ g^2_i\rangle _{U_2}g^2_i.$$
Consider the collections of $\sigma$-algebras on $(\Omega,\sF,P)$ given by
$$ \sF^n_t=\sigma(W^n_s,0\leq s\leq t)\vee\sF_{t,T}^B.   $$
We put, by definition, for any $t\in [0,T]$
\begin{equation}
  \begin{split}
    &\tilde{F}^n(t,\phi,\varphi)=E\left[F(t,\phi,\varphi)|\sF_t^n   \right],
    F^n(t,\phi,\varphi)=P_n \tilde{F}^n(t,\phi,\varphi) \textrm{ and }  \\
    &\tilde{J}^n(t,\phi,\varphi)=E\left[J(t,\phi,\varphi)|\sF_t^n   \right],
    J^n(t,\phi,\varphi)=P_n \tilde{J}^n(t,\phi,\varphi),\ \phi\in V,\varphi\in L(U_1,H).
  \end{split}
\end{equation}
For
each $n\in \mathbb{N}$ we consider the following backward doubly stochastic
differential equation on $H_n:$
\begin{equation}\label{bdsdes finite dim sys}
  \begin{split}
    u^n(t)=&\
        G^n+\int_t^T\!\!\!\! F^n(s,u^n(s),v^n(s))\ ds
        +\int_t^T\!\!\!\!J^n(s,u^n(s),v^n(s))\ dB^n_s\\
        &\ -\int_t^T\!\!\!\!v^n(s)\ dW^n_si
  \end{split}
\end{equation}
with $G^n:=E\left[P_n G|\sF_T^n  \right]$.

\begin{proof}[Proof of Theorem \ref{thm main}]
    The uniqueness follows from Lemma \ref{lem uniqueness} and it remains to prove the existence and estimate \eqref{proof of thm main estimate 1}. First, for every $n\in \bN$ it can be checked that \eqref{bdsdes finite dim sys} satisfies all the conditions of Theorem \ref{thm finite dim} only with $\varsigma(t)$ replaced by $\varsigma^n(t):=E[\varsigma(t)|\sF^n_t]$.
  In view of Theorem \ref{thm finite dim} and Lemma \ref{rmk thm finite dim estimate}, there exists a unique solution $(u^n,v^n)\in (S^p(0,T;H)\cap M^{pq/2,q}(0,T;V))\times M^{p,2}(0,T;L(U_1,H))$ to \eqref{bdsdes finite dim sys} such that
  \begin{equation}\label{proof of thm main estimate 1}
    \begin{split}
      &\|u^n\|_{S^{p}(0,T;H)}+\|v^n\|_{M^{p,2}(0,T;L(U_1,H))}+\|u^n\|_{M^{pq/2,q}(0,T;V)}^{q/2}\\
      \leq&\
      C\left\{ \|\varsigma^n\|^{1/2}_{M^{p/2}(0,T)}+\|G^n\|_{L^p(\Omega,\sF_F,H)} \right\}
      \\
      \leq&\
      C\left\{ \|\varsigma\|^{1/2}_{M^{p/2}(0,T)}+\|G\|_{L^p(\Omega,\sF_F,H)} \right\}.
    \end{split}
  \end{equation}
  On the other hand,
\begin{equation*}
  \begin{split}
        &E\left[\int_0^T\left(\|F^n(t,u^{n}(t),v^n(t))\|_*^{q'}+\|\tilde{F}^n(t,u^{n}(t),v^n(t))\|_*^{q'}\right)\ dt  \right]\\
    \leq &\ C
        E\left[\int_0^T\|F(t,u^n(t),v^n(t))\|_*^{q'}\ dt  \right]\\
    \leq &\ C
        E\left[\int_0^T\left(\varsigma(t)+K(\|u^{n}(t)\|_V^q+\|u^n(t)\|^2+\|v^n(t)\|_1^2)\right)
        \left(1+\|u^{n}(t)\|^{\beta}\right)\ dt   \right]
            \\
    \leq&\ C
        \left\{\|\varsigma\|_{M^{p/2,1}(0,T)}^{p/2}+\|v^n\|_{M^{p,2}(0,T;L(U_1,H))}^p
        +\|u^{n}\|^p_{S^p(0,T;H)}+\|u^n\|^{pq/2}_{M^{pq/2,q}(0,T;V)}
        \right\}\\
  \end{split}
\end{equation*}
  and
\begin{equation*}
  \begin{split}
    & E\left[\int_0^T\left(\|J^n(t,u^{n}(t),v^{n}(t))\|_2^2+\|\tilde{J}^n(t,u^{n}(t),v^{n}(t))\|_2^2\right)\ dt\right]\\
     \leq&\
        CE\left[\int_0^T\left(\varsigma(t)+\|u^{n}(t)\|_V^q
        +\|u^n(s)\|^2+\|v^{n}(t)\|_1^2\right) \ dt\right]\\
     \leq&\ C
        \left\{\|\varsigma\|_{M^{p/2,1}(0,T)}+\|v^n\|_{M^{p,2}(0,T;L(U_1,H))}^2
        +\|u^{n}\|^2_{S^p(0,T;H)}+\|u^n\|^{q}_{M^{pq/2,q}(0,T;V)}
        \right\},\\
  \end{split}
\end{equation*}
where the constants $C$s are all independent of $n$.
  Thus, there exists a positive constant $C$ independent of $n$ such that
  \begin{equation}
    \begin{split}
      C\geq &\ \|u^{n}\|^p_{S^p(0,T;H)}+\|u^n\|^{pq/2}_{M^{pq/2,q}(0,T;V)}
      +\|v^n\|_{M^{p,2}(0,T;L(U_1,H))}^p\\
      &+\|J^n(t,u^{n}(t),v^{n}(t))\|^p_{M^{2}(0,T;L(U_2,H))}
      +\|F^n(t,u^{n}(t),v^{n}(t))\|^{q'}_{M^{q'}(0,T;V')}\\
      &+\|\tilde{J}^n(t,u^{n}(t),v^{n}(t))\|^p_{M^{2}(0,T;L(U_2,H))}
      +\|\tilde{F}^n(t,u^{n}(t),v^{n}(t))\|^{q'}_{M^{q'}(0,T;V')}
    \end{split}
  \end{equation}
  from which it follows that there exists a subsequence $n_k\rightarrow\infty$ and $(u,v,\bar{F},\bar{J},\tilde{F},\tilde{J})$ such that
  \begin{equation*}
    \begin{split}
    u^{n_k}&\longrightarrow u \textrm{ weakly star in } L^{p}(\Omega,L^{\infty}(0,T;H));\\
    u^{n_k}&\longrightarrow u \textrm{ weakly in } M^{pq/2,q}(0,T;V);\\
    v^{n_k}&\longrightarrow v \textrm{ weakly in } M^{p,2}(0,T;L(U_1,H));\\
    F^{n_k}(\cdot,u^{n_k}(\cdot),v^{n_k}(\cdot))&\longrightarrow \bar{F} \textrm{ weakly in } M^{q'}(0,T;V);\\
    J(\cdot,u^{n_k}(\cdot),v^{n_k}(\cdot))&\longrightarrow \bar{J} \textrm{ weakly in } M^{2}(0,T;L(U_2,H));\\
    \tilde{F}^{n_k}(\cdot,u^{n_k}(\cdot),v^{n_k}(\cdot))&\longrightarrow \tilde{F} \textrm{ weakly in } M^{q'}(0,T;V);\\
    \tilde{J}(\cdot,u^{n_k}(\cdot),v^{n_k}(\cdot))&\longrightarrow \tilde{J} \textrm{ weakly in } M^{2}(0,T;L(U_2,H)).
    \end{split}
  \end{equation*}
  Through a density argument, we check that $(\tilde{F},\tilde{J})\equiv (\bar{F},\bar{J})$.

  Through such a  calculation as
  \begin{equation*}
    \begin{split}
        &\lim_{n\rightarrow\infty}E\Big[\|G-E\left[P_n G|\sF_T^n\right]\|^p\Big]\\
    \leq&\
        2^{p-1}\lim_{n\rightarrow\infty}E\Big[\|G-E\left[G|\sF_T^n\right]\|^p
        +\|E\left[G-P_n G|\sF_T^n\right]\|^p\Big]\\
    \leq &\
        2^{p-1}\lim_{n\rightarrow\infty}E\Big[\|G-E\left[G|\sF_T^n\right]\|^p
        +\|G-P_n G\|^p\Big]=0,
    \end{split}
  \end{equation*}
  we obtain
  \begin{equation*}
    \begin{split}
      &G^{n_k}\longrightarrow G \textrm{ strongly in } L^p(\Omega,\sF_T,H)\\
      \textrm{and}&\textrm{ for }dt\textrm{-almost all }  t\in [0,T],\ \forall(\varphi,\xi)\in V\times L(U_1,H),\\
      &\lim_{n\rightarrow \infty}E[\|\tilde{F}^n(t,\varphi,\xi)-F(t,\varphi,\xi)\|_*^{q'}
        +\|\tilde{J}^n(t,\varphi,\xi)-J(t,\varphi,\xi)\|_2^2]=0.
    \end{split}
  \end{equation*}
  Then by Lemma \ref{lem in pre variation}, Remark \ref{rmk lem in pre variation} and \ref{rmk-lem-pre-variation}, $(u,v)$ is the unique solution of BDSDES \eqref{bdsdes}.
  Moreover, from \eqref{proof of thm main estimate 1} we deduce that estimate \eqref{main thm estimate} holds.
  We complete the proof.
\end{proof}

\section{Examples}
First,
let us consider the following quasi-linear BDSPDE:
\begin{equation}\label{1.1}
  \left\{\begin{array}{l}
  \begin{split}
  -du(t,x)=\,&\displaystyle \biggl[\partial_{x_j}\Bigl(a^{ij}(t,x)\partial_{x_i} u(t,x)
        +\sigma^{jr}(t,x) v^{r}(t,x)     \Bigr) +b^j(t,x)\partial_{x_j}u(t,x)\\
        &\displaystyle
         +c(t,x)u(t,x)+\varsigma^{r}(t,x)v^r(t,x)+g(t,x,u(t,x),\nabla u(t,x),v(t,x))\\
        &\displaystyle +\partial_{x_j}f^j(t,x,u(t,x),\nabla u(t,x),v(t,x))
                \biggr]\, dt-v^{r}(t,x)\, dW_{t}^{r}\\ &\displaystyle
           +h^l((t,x,u(t,x),\nabla u(t,x),v(t,x))\ d\overrightarrow{B}^l_t, \
                     (t,x)\in Q:=[0,T]\times \mathcal{O};\\
    u(T,x)=\,& G(x), \quad x\in\cO.
    \end{split}
  \end{array}\right.
\end{equation}
Here and in the following we use Einstein's summation convention, $T\in(0,\infty)$  is a fixed deterministic terminal time, $\cO\subset \bR^n$
is a domain with boundary $\partial \cO \in C^1$, $\nabla=(\partial_{x_1},\cdots,\partial_{x_n})$ is the gradient operator in $\bR^n$, and
$\left\{W_t:=(W^1_t, \cdots, W^m_t),  t\in [0,T]\right\}$ and $\left\{ B_t:=(B^1_t, \cdots, B^m_t), t\in [0,T]\right\}$ are two mutually independent $m$-dimensional standard Brownian motions. Note that domain $\cO$ can be chosen to be the whole space $\bR^n$.

To BDSPDE \eqref{1.1}, we give the following assumptions.

\bigskip\medskip
   $({\mathcal B} 1)$ \it The triple
\begin{equation*}
  (f,g,h)(\cdot,t,\cdot,\vartheta,y,z):~\Omega\times[0,T]\times\cO\rightarrow\bR^n\times\bR\times\bR^{l}
\end{equation*}
are $\sF_t\otimes\cB(\cO)$-measurable for any $(t,\vartheta,y,z)\in [0,T]\times\bR\times\bR^{n}\times\bR^{ m}$. There exist nonnegative constants $\delta\in (0,1),\kappa,\alpha,\beta$
and $L$
   such that for all $(\vartheta_1,y_1,z_1),(\vartheta_2,y_2,z_2)\in \bR\times\bR^n\times\bR^{n\times m}$
   and $(\omega,t,x)\in \Omega\times[0,T]\times\cO$,
   \begin{equation*}
     \begin{split}
       |f(\omega,t,x,\vartheta_1,y_1,z_1)-f(\omega,t,x,\vartheta_2,y_2,z_2)|\leq&\  L|\vartheta_1-\vartheta_2|+\frac{\kappa}{2}|y_1-y_2|+\beta^{1/2}|z_1-z_2|,\\
       |g(\omega,t,x,\vartheta_1,y_1,z_1)-g(\omega,t,x,\vartheta_2,y_2,z_2)|\leq&\  L(|\vartheta_1-\vartheta_2|+|y_1-y_2|+|z_1-z_2|),\\
       |h(\omega,t,x,\vartheta_1,y_1,z_1)-h(\omega,t,x,\vartheta_2,y_2,z_2)|\leq&\  L|\vartheta_1-\vartheta_2|+{\left(\alpha\right)^{1/2}}\!\!\!|y_1-y_2|
       +\left(\delta\right)^{1/2}\!\!\!\!\!|z_1-z_2|.
     \end{split}
   \end{equation*}\rm

\medskip
   $({\mathcal B}2)$ \it For each $t\in[0,T]$, the functions $a(t),\,\sigma(t),\, b(t),\,c(t),\,\varsigma(t)$ are $\sF_t\otimes\cB(\cO)$-measurable. There exist constants $\varrho,\varrho'> 1,$ and $\lambda,\Lambda> 0$ such that
   the following  hold for all $\xi\in\bR^n$ and $(\omega,t,x)\in \Omega\times[0,T]\times\cO$,
   \begin{equation*}
     \begin{split}
       &\lambda|\xi|^2\leq (2a^{ij}(\omega,t,x)-\varrho\sigma^{ir}\sigma^{jr}(\omega,t,x))\xi^i\xi^j\leq \Lambda|\xi|^2;\\
       &|a(\omega,t,x)|+|\sigma(\omega,t,x)|+|b(\omega,t,x)|+|c(\omega,t,x)|+|\varsigma(\omega,t,x)|\leq \Lambda;\\
       &\hbox{ \rm and  }\lambda-\kappa-\varrho'\beta-\alpha>0 \textrm{ \rm with }
       \frac{1}{\varrho}+\frac{1}{\varrho'}+\delta=1.
     \end{split}
   \end{equation*}\rm

\medskip
   $({\mathcal B} 3)$ \it $G\in L^{2}(\Omega,\sF_T,L^2(\cO))$,  and
   $h_0:=h(\cdot,\cdot,\cdot,0,0,0)\in M^2(0,T;L^2(\cO;\bR^m)),$
\begin{equation*}
    \begin{split}
    f_0:=f(\cdot,\cdot,\cdot,0,0,0)\in M^2(0,T;L^2(\cO;\bR^n)),\ g_0:=g(\cdot,\cdot,\cdot,0,0,0)\in M^2(0,T;L^2(\cO)).
    \end{split}
\end{equation*}

Let $C_c^{\infty}(\cO)$ denote the set of all infinitely differentiable real-valued functions on $\cO$ with compact support. For $p\in [1,\infty)$ and $\phi\in C_c^{\infty}(\cO)$, define
\begin{equation}
  \|u\|_{H^{1,p}_0(\cO)}:=\left( \int_{\cO} \left( |\phi(x)|^p+|\nabla\phi(x)|^p  \right)dx\right)^{1/p}.
\end{equation}
Then the Sobolev space  $H^{1,p}_0(\cO)$ is defined as the completion of $C_c^{\infty}(\cO)$ with respect to the norm $\|\cdot\|_{H^{1,p}_0(\cO)}$. As usual, we denote $H^{1,2}_0(\cO)$ and its dual space $H^{-1,2}(\cO)$ by $H^1_0(\cO)$ and $H^{-1}(\cO)$ respectively.
Here, Gelfand triple $(V,H,V')$ is realized as the triple $(H_0^1(\cO),L^2(\cO),H^{-1}(\cO))$. Hence, by Theorem \ref{thm main}, we have
\begin{prop}\label{prop}
  Let the assumptions $(\mathcal{B}1)$-$(\mathcal{B}3)$ hold. Then BDSPDE \eqref{1.1} admits a unique solution $(u,v)\in \left(M^2(0,T;L^2(\cO))\cap M^2(0,T;H_0^1(\cO))\right)\times M^2(0,T;L^2(\cO;\bR^m))$
  which satisfies
  \begin{equation*}
    \begin{split}
      &\|u\|_{S^2(0,T;L^2(\cO))}+\|u\|_{M^2(0,T;H_0^2(\cO))}+\|v\|_{M^2(0,T;L^2(\cO;\bR^m))}\\
      \leq&\
            C\left\{
            \|G\|_{L^2(\Omega,\sF_T,L^2(\cO))}
            +\|f_0\|_{M^2(0,T;L^2(\cO;\bR^n))}+\|g_0\|_{M^2(0,T;L^2(\cO))}
            +\|h_0\|_{M^2(0,T;L^2(\cO;\bR^m))}
            \right\}
    \end{split}
  \end{equation*}
  with the constant $C$ depending on $\lambda,\alpha,\beta,\delta,L,\kappa,\varrho,\Lambda,L$ and $T$.
\end{prop}
\begin{rmk}
  In Proposition \ref{prop}, if we assume further that $h\equiv 0$, $G$ is $\sF_T^W$-measurable, and
  for any $(\vartheta,y,z)\in \bR\times\bR^{n}\times\bR^{ m}$ $f(\cdot,\vartheta,y,z)$, $g(\cdot,\vartheta,y,z)$ and $h(\cdot,\vartheta,y,z)$ are all $\sF_t^W$-adapted processes. Then our BDSDES \eqref{1.1} degenerates into a BSPDE on which
  some behavior properties of the solutions, on basis of Proposition \ref{prop}, have been obtained by Qiu and Tang \cite{QiuTangMPBSPDE11}
   under a more general framework.
\end{rmk}
\begin{rmk}
  In view of the whole proof of our main theorem \ref{thm main}, we deal in fact with a much more general class of BDSPDEs. Precisely, we solve the following BDSPDE:
  \begin{equation}
    \begin{split}
      u(t,x)
      =
      &\
        G(x)+\int_t^T\!\! \cL u(s,x)+\langle\delta(s,x), Dv(s,x)\rangle+f(s,x,u(s,x),D u(s,x),v(s,x)) \ ds
      \\
      &\
        +\int_t^T\!\! h^r(s,x,u(s,x),D u(s,x),v(s,x)) \ d\overrightarrow{B}^r_s
        -\int_t^T\!\! v^r(s,x)\ dW^r_s
    \end{split}
  \end{equation}
  where $\esssup_{s\in[0,T],x\in\mathcal{Q}}\|\delta(x)\|\leq c_0<1$ and $\cL$ is a non-positive self-adjoint sub-Markovian operator associated with a symmetric Dirichlet form defined on some space $L^2(\mathcal {Q},m(dx))$ and which admits a gradient $D$. One particular case of the previous example lies in the case where $\mathcal{Q}=\cO$, $m(dx)=dx$, $Du(x)=\nabla u(x)\sigma(x)$
  and
  $$\cL u(t,x)=\partial_{x_j}\left(a_{ij}(x)\partial_{x_i} u(t,x) \right)$$
  with $a=\sigma\sigma^*\geq 0$ which is not necessary to be uniformly positive definite as $(\mathcal{B}2)$ and is allowed to be degenerate. We refer to \cite{BouleauHirschDirichletform94,FukushimaDirichletform94,Ma_Rockner_Dirichlet} for a detailed exposition and references to the theory of Dirichlet forms. We also refer to \cite{DenisDirichlet2004,Denis2004} for a counterpart on the SPDE theory.
\end{rmk}

It is worthy noting that our BDSDESs like \eqref{bdsdes} include as particular cases  the forward stochastic differential evolutionary systems listed in \cite[Chapter 4, Page 55--91]{PrevotRockner2007}.
Consider the following BDSDES:
\begin{equation}
  \begin{split}
    u(t)=
    &
        \ G+\int_t^T\!\! A(s,u(s))+\delta_1 v(s)+f(s)\ ds\\
    &
        \ +\int_t^T\!\! g(s)+\delta_2 v(s)+A_1(s,u(s)) \ d\overrightarrow{B}_s
        -\int_t^T\!\! v(s)\ dW_s
  \end{split}
\end{equation}
with $W$ and $B$ being one-dimensional Wiener processes and
$(\delta_1,\delta_2)\in\bR\times(0,1)$. Let $A_1(t,u)$ be Lipschitzian continuous with respect to $u$ on $H$.
Then $A(t,u)$ can be chosen to be any one listed in \cite[Chapter 4, Page 59--73]{PrevotRockner2007} with corresponding Gelfand triple. For example,

$\tbf{(a)}$. $A(u):=-u|u|^{r-2}$ with $(V,H,V'):=(L^r(\cO),L^2(\cO),L^{r/(r-1)}(\cO))$ and $r\in[2,\infty)$;

$\tbf{(b)}$. $A(u):=\textrm{div}\left(|\nabla u|^{r-2}\nabla u\right)$ with $(V,H,V'):=(H^{1,r}_0(\cO),L^2(\cO),(H^{1,r}_0(\cO))')$ and $r\in[2,\infty)$.
Then the corresponding existence and uniqueness propositions  are implied by Theorem \ref{thm main}.

\section{Appendix}
As in \cite{Krylov_Rozovskii81} and
\cite{PrevotRockner2007,RenRocknerWang2007}, to prove Theorem \ref{thm ito
formula} we need the following lemma. For abbreviation below we set
$$\mathbb{X}:=L^q(\Omega\times[0,T],\sF\otimes\cB([0,T]),\mathbb{P}\otimes dt;V).$$
\begin{lem}\label{lem ito formula}
    Let $u\in \mathbb{X}$. Then there exists a sequence of partitions
    $I_l:=\{0=t_0^l<t_1^l <\cdots<t^l_{k_l}=T\}$ such that $I_l\subset
    I_{l+1}$ and $\pi(I_l):=\max_i(t_i^l-t^l_{i-1})\rightarrow 0$, $u(t_i^l)\in
    V$ $P$-a.e. for all $l\in \bN, 1\leq i\leq k_l-1$, and for
    $$
    \bar{u}^l:=\sum_{i=2}^{k_l} 1_{[t_{i-1}^l,t_i^l[} u(t_{i-1}^l),
    ~~\tilde{u}^l:=\sum_{i=1}^{k_l-1}1_{[t_{i-1}^l,t_i^l[} u(t_{i}^l),
    ~~l\in \bN,
    $$
    we have $\bar{u}^l$ and $\tilde{u}^l$ belong to $\mathbb{X}$ such that
    $$\lim_{l\rightarrow\infty}\{ \|u-\bar{u}^l\|_{\mathbb{X}}+\|u-\tilde{u}^l\|_{\mathbb{X}} \}=0.$$
\end{lem}
Since the proof of Lemma \ref{lem ito formula} is standard (c.f.
\cite[Lemma 4.2.6]{PrevotRockner2007} or \cite[Lemma 4.1]{RenRocknerWang2007}), we omit it here.
\begin{proof}[Proof of Theorem \ref{thm ito formula}]
    \tbf{Step 1}.
    Obviously, Equation \eqref{BDSES
    trivialcase} holds on $V'$, i.e. \eqref{BDSES trivialcase} holds with both sides being $V'$-valued
     processes. Denote
    $$\tilde{W}_t:=\int_0^t v(r)\ dW_r \textrm{ and } \tilde{B}_t:=\int_0^t h(r)\ d\overrightarrow{B}_r,~~t\in [0,T].$$
    Then $\tilde{W}$ and $\tilde{B}$ are continuous $H$-valued processes. Since $f\in \mathbb{X}':= M^{q'}(0,T;V')$, both $\int_0^tf(r)dr$ and $u$ are
    continuous $V'$-valued processes.
    Through careful computations, we have
    \begin{equation}\label{proof of ito formula indentity}
      \begin{split}
        \|u(s)\|^2=
        &\|u(t)\|^2-\|\tilde{W}_t-\tilde{W}_s\|^2+\|\tilde{B}_t-\tilde{B}_s\|^2
        +2\int_s^t\!\langle u(s),\ f(r)\rangle \ dr\\
        &-2\int_s^t\!\!\langle u(s),\ v(r)\ dW_r\rangle +2\int_s^t\!\!\langle u(t),\ h(r)\ d\overrightarrow{B}_r\rangle \\
        &-\|u(t)-u(s)-\tilde{W}_t+\tilde{W}_s+\tilde{B}_t-\tilde{B}_s\|^2\\
        &-2\langle u(t)-u(s)-\tilde{W}_t+\tilde{W}_s+\tilde{B}_t-\tilde{B}_s,
                    \tilde{W}_t-\tilde{W}_s\rangle \\
      \end{split}
    \end{equation}
    holds for all $t,s\in [0,T]$ such that $t>s$ and $u(t),u(s)\in V$.
    For any $t=t_i^l\in I_l\backslash \{ 0,T\}$ given in Lemma \ref{lem ito
    formula}, we have
    \begin{align}
        &\|u(t)\|^2-\|\xi\|^2\nonumber\\
        =&
            \sum_{j=i+1}^{k_l}(\|u(t_{j-1}^l)\|^2-\|u(t_{j}^l)\|^2)\nonumber\\
        =&
            2\int_t^T\!\!\langle f(r),\ \bar{u}^l(r)\rangle \ dr
            -2\int_t^T\!\!\langle \bar{u}^l(r),\ v(r)dW_r\rangle
            +2\int_t^T\!\!\langle \tilde{u}^l(r),\ h(r)\ d\overrightarrow{B}_r\rangle \nonumber\\
        &
            +2\langle \xi,\ \int_{t_{k_l-1}^l}^T \!\!h(r)\ d\overrightarrow{B}_r\rangle
            +\sum_{j=i+1}^{k_l}\left(\|\tilde{B}_{t_{j}^l}-\tilde{B}_{t_{j-1}^l}\|^2
                             -\|\tilde{W}_{t_j^l}-\tilde{B}_{t_{j-1}^l}\|^2\right)\nonumber\\
        &+\sum_{j=i+1}^{k_l}\Big(-
        2\langle u(t_{j}^l)-u(t_{j-1}^l)-\tilde{W}_{t_{j}^l}+\tilde{W}_{t_{j-1}^l}
        +\tilde{B}_{t_{j}^l}-\tilde{B}_{t_{j-1}^l},\
                    \tilde{W}_{t_{j}^l}-\tilde{W}_{t_{j-1}^l}\rangle \nonumber\\
        &\quad\quad\quad
         \textrm{   }-\|u(t_{j}^l)-u(t_{j-1}^l)-\tilde{W}_{t_{j}^l}+\tilde{W}_{t_{j-1}^l}
        +\tilde{B}_{t_{j}^l}-\tilde{B}_{t_{j-1}^l}\|^2\Big).
    \end{align}
    It can be checked that all the integral above are well defined.
    By Lemma \ref{lem ito formula}, there holds
    \begin{equation}\label{proof of ito formula 1}
        E\left[\int_0^T\!\!|\langle \bar{u}^l(s),\ f(s)\rangle |\ ds\right]
        \leq\|f\|_{\mathbb{X'}}\|\bar{u}^l\|_{\mathbb{X}} < c_1,
    \end{equation}
    where the constant $c_1>0$ is independent of $l$. By BDG
    inequality, we have
    \begin{equation}\label{proof of ito formula 2BDG}
      \begin{split}
        &
            E\left[ \sup_{t\in [0,T]}\left| \int_t^T\!\!\langle \bar{u}^l(s),\ v(s)\ dW_s\rangle
            -\int_t^T\!\!\langle \tilde{u}^l(s),\ h(s)\ d\overleftarrow{B}_s\rangle  \right|  \right]\\
        \leq&\
            2E\left[ \sup_{t\in [0,T]}\left|
            \int_0^t\!\!\langle \bar{u}^l(s),\ v(s)\ dW_s\rangle \right|\right]
            +E\left[\sup_{t\in[0,T]}\left|\int_t^T\!\!\langle \tilde{u}^l(s),\ h(s)\ d\overleftarrow{B}_s\rangle  \right|
            \right]\\
        \leq&\
            C E\left[\left|\int_0^T\left( \|v(s)^*\bar{u}^l(s)\|_{U_1}^2
            +\|h(s)^*\tilde{u}^l(s)\|^2_{U_2}\right)\ ds\right|^{1/2}  \right]\\
        \leq&\ C
            E\left[\left|\int_0^T\left( \|v(s)\|_1^2\|\bar{u}^l(s)\|^2
            +\|h(s)\|_2^2\|\tilde{u}^l(s)\|^2\right)\ ds\right|^{1/2}  \right]\\
        \leq&\
            \frac{1}{4}E\left[\sup_{1\leq j\leq k_l}\|u(t_{j}^l)\|^2\right]
            +CE\left[\int_0^T\left(\| h(s) \|_2^2+\|v(s)\|_1^2\right)\ ds
        \right],
      \end{split}
    \end{equation}
    with $C$ being a generic constant independent of $l$.
    On the other hand, we have
    \begin{align}
        &
            \sum_{j=i+1}^{k_l}\Big(
               -2\langle u(t_{j}^l)-u(t_{j-1}^l)-\tilde{W}_{t_{j}^l}
                +\tilde{W}_{t_{j-1}^l}
                +\tilde{B}_{t_{j}^l}-\tilde{B}_{t_{j-1}^l},\
                \tilde{W}_{t_{j}^l}-\tilde{W}_{t_{j-1}^l}\rangle \nonumber\\
        &\ \ \ \ \ \
            -\|u(t_{j}^l)-u(t_{j-1}^l)-\tilde{W}_{t_{j}^l}
            +\tilde{W}_{t_{j-1}^l}
            +\tilde{B}_{t_{j}^l}-\tilde{B}_{t_{j-1}^l}\|^2\Big)\nonumber\\
        \leq&\
            \sum_{j=i+1}^{k_l}\Big[
            \|u(t_{j}^l)-u(t_{j-1}^l)-\tilde{W}_{t_{j}^l}
            +\tilde{W}_{t_{j-1}^l}
            +\tilde{B}_{t_{j}^l}
            -\tilde{B}_{t_{j-1}^l}\|^2
            +\|\tilde{W}_{t_{j}^l}-\tilde{W}_{t_{j-1}^l}\|^2\nonumber\\
        &\ \ \ \ \ \
            -\|u(t_{j}^l)-u(t_{j-1}^l)-\tilde{W}_{t_{j}^l}
            +\tilde{W}_{t_{j-1}^l}
            +\tilde{B}_{t_{j}^l}
            -\tilde{B}_{t_{j-1}^l}\|^2      \Big]\nonumber\\
        =&\
            \sum_{j=i+1}^{k_l}\left(
            \|\tilde{W}_{t_{j}^l}
            -\tilde{W}_{t_{j-1}^l}\|^2
            \right)                              \nonumber\\
        =&\
        \int_{t_i^l}^T\!\|v(s)\|_1^2\ ds, \label{proof of ito formula 3}\\
\nonumber\\
        &
            E\left[ \sum_{j=i+1}^{k_l}\left(
            -\|\tilde{W}_{t_j^l}
            -\tilde{W}_{t_{j-1}^l}\|^2
            +\|\tilde{B}_{t_j^l}-\tilde{B}_{t_{j-1}^l}\|^2\right) \right]\nonumber\\
        =&\
            \sum_{j=i+1}^{k_l} E \left[ \int_{t_{j-1}^l}^{t_j^l}
            \left(-\|v(s)\|_1^2+\|h(s)\|_2^2\right)\ ds \right]\nonumber\\
        =&\
            E\left[\int_{t_i^l}^T\left(-\|v(s)\|_1^2+\|h(s)\|_2^2 \right)\ ds
            \right]    \label{proof of ito formula 5}
    \end{align}
    and
    \begin{equation}\label{proof of ito formula 4}
        E\left[\left|(\xi,\ \int_{t_{k_l-1}^l}^T\!\! h(s)\ dW_s )\right|\right]
        \leq \left(E\left[\|\xi\|^2\right]\right)^{1/2}\left(
        E\left[\int_{t_{k_l-1}^l}^T\!\!
        \|h(s)\|_2^2\ ds\right]\right)^{1/2}.
    \end{equation}

    Hence, in view of
    \eqref{proof of ito formula 1}-\eqref{proof of ito formula
    4}, we obtain
    $$E\left[\sup_{t\in I_l\backslash\{0\}}\|u(t)\|^2\right]\leq c_2<\infty$$
    for some constant $c_2>0$ independent of $l$. Therefore, setting $I:=\cup_{l\geq 1}I_l\backslash\{0\}$,
    with $I_l$ as in Lemma \ref{lem ito formula}, we have
    $$E\left[\sup_{t\in I}\|u(t)\|^2\right]\leq c_2,$$
    since $I_l\subset I_{l+1}$ for all $l\in \bN$. Almost surely,
    $$ \sum_{j=1}^N|\langle u(\omega,t),e_j\rangle |^2\  \uparrow \|u(\omega,t)\|^2\textrm{ as }  N\uparrow \infty, \ \forall t\in[0,T]$$
    with $\{ e_j\big|j\in\bN\}\subset V$ being  an orthonormal basis of $H$. For any
     $x\in V'\backslash H$, set $\|x\|=\infty$ as usual. Then, we conclude
    that $t\rightarrow \|u(t)\|$ is lower semicontinuous almost surely.
    Since $I$ is dense in $[0,T]$, we arrive at $\sup_{t\in[0,T]}\|u(t)\|^2
    =\sup_{t\in I}\|u(t)\|^2$. Hence, we have
    \begin{equation}\label{proof of ito formula sup norm}
          E\left[\sup_{t\in[0,T]}\|u(t)\|^2 \right]<\infty.
    \end{equation}

    \tbf{Step 2}. We prove the following approximating result:
    \begin{equation}\label{proof of ito formula claim a}
        \begin{split}
            \lim_{l\rightarrow \infty}\sup_{t\in [0,T]}\left|
            \int_t^T\!\langle u(s)-\bar{u}^l(s),\ v(s)\ dW_s\rangle
            \right|=0\textrm{ in probability,}\\
            \lim_{l\rightarrow \infty}\sup_{t\in [0,T]}\left|
            \int_t^T\!\langle u(s)-\tilde{u}^l(s),\ h(s)\ d\overleftarrow{B}_s\rangle
            \right|=0\textrm{ in probability.}
        \end{split}
    \end{equation}

    As to \eqref{proof of ito formula claim a}, it is sufficient to prove the first equality, since the
    second follows similarly. As $u$ is a continuous $V'$-valued process, we conclude
    from \eqref{proof of ito formula sup
    norm}  that $u$ is weakly continuous in $H$. It follows that
    $P_n u$ is continuous in $H$ and thus
    $$ \lim_{l\rightarrow\infty}\int_0^T\|P_n (u(s)-\bar{u}^l(s))\|^2\|v(s)\|_1^2ds=0,~a.s.,
     $$
     where $P_n$ is the orthogonal projection onto span$\{e_1,\dots,e_n\}$
     in $H$. It remains to prove that for each $\eps>0$,
     \begin{equation}\label{proof of ito formula claim a1}
       \begin{split}
	       &\lim_{n\rightarrow \infty}\sup_{l\in \bN} \mathbb{P}\left(
         \sup_{t\in[0,T]}\left|
              \int_t^T \!\langle (1-P_n)\bar{u}^l(s),\ v(s)\ dW_s\rangle
                    \right|> \eps   \right)=0,\\
		    &\lim_{n\rightarrow \infty}\mathbb{P} \left(
         \sup_{t\in[0,T]}\left|
              \int_t^T\!\! \langle (1-P_n)u(s),\ v(s)\ dW_s\rangle
                    \right|>\eps   \right)=0.
       \end{split}
     \end{equation}
    For each $n\in \bN$, $\gamma\in (0,1)$ and $N>1$,
     we have
    \begin{align*}
        &
	\mathbb{P}\left(\sup_{t\in[0,T]}\left| \int_t^T\!\!\langle (1-P_n)\bar{u}^l(s),\ v(s)dW_s
            \rangle  \right|> \eps  \right)                                                \\
      \leq&\
      \mathbb{P}\left(\sup_{t\in[0,T]}\left| \int_0^t\!\!\langle (1-P_n)\bar{u}^l(s),\ v(s)dW_s
            \rangle  \right|> \frac{\eps }{2} \right)                                     \\
      =&\
      \mathbb{P}\left(\sup_{t\in[0,T]}\left| \int_0^t\!\!\langle (1-P_n)\bar{u}^l(s),\ v(s)dW_s
            \rangle  \right|> \frac{\eps }{2},
            \int_0^T\|\bar{u}^l(s)\|^2d\langle(1-P_n)\tilde{W}\rangle_s\leq \gamma^2\right)            \\
      &\
      +\mathbb{P}\left(\sup_{t\in[0,T]}\left| \int_0^t\!\!\langle (1-P_n)\bar{u}^l(s),\ v(s)dW_s
            \rangle  \right|> \frac{\eps }{2},
            \int_0^T\!\!\!\|\bar{u}^l(s)\|^2\ d\langle(1-P_n)\tilde{W}\rangle_s> \gamma^2\right)                \\
      \leq&\
      \mathbb{P}\left(\sup_{t\in[0,T]}\left| \int_0^t\!\!\langle (1-P_n)\bar{u}^l(s),\ v(s)dW_s
            \rangle  \right|> \frac{\eps }{2},
            \int_0^T\|\bar{u}^l(s)\|^2d\langle(1-P_n)\tilde{W}\rangle_s\leq \gamma^2\right)            \\
      &\
      +\mathbb{P}\left(\left|
            \int_0^T\!\!\!\|\bar{u}^l(s)\|^2\ d\langle(1-P_n)\tilde{W}\rangle_s\right|  > \gamma^2\right)\\
      \leq&\
            \frac{2C}{\eps}E\left[\left|\int_0^T\!\!\!\|\bar{u}^l(s)\|^2\ d\langle(1-P_n)\tilde{W}\rangle_s\right|^{1/2}
            \wedge\gamma\right]
	    +\mathbb{P}\left(\left|
            \int_0^T\!\!\!\|\bar{u}^l(s)\|^2\ d\langle(1-P_n)\tilde{W}\rangle_s\right|  > \gamma^2\right)\\
       \leq&
            \frac{2C\gamma}{\eps}
	    +\mathbb{P}\left(\sup_{t\in[0,T]}\|u(t)\|>N  \right)
            +\frac{N^2}{\gamma^2}E\left[\langle(1-P_n)\tilde{W}\rangle_T\right],
    \end{align*}
    where $C$ is a constant from BDG inequality and
     $\langle(1-P_n)\tilde{W}\rangle_t:=\int_0^t\|(1-P_n)v(s)\|_1^2ds$.
     By letting $n\rightarrow \infty$, then $N\rightarrow \infty$ and
     finally $\gamma\rightarrow 0$, we complete the proof of the first equality of \eqref{proof of ito formula claim a1}. The second equality of \eqref{proof of ito formula claim a1} follows similarly.

    \tbf{Step 3}.
       We prove \eqref{Ito formula square} holds for $t\in I$.

    For this $t\in I$ fixed, we may assume that $t\neq T$. In this case,
    there exists a $N\in \bN$ such that for any $l\geq N$, there exists a
    unique $0<i<k_l$ satisfying $t=t_i^l$. In view of \eqref{proof of ito formula 5}, \eqref{proof of ito formula 4}, \eqref{proof of ito formula claim
    a} and Lemma \ref{lem ito formula}, taking limits in probability, we have
    \begin{equation}
      \begin{split}
        &\|u(t)\|^2-\|\xi\|^2\\
        =&2\int_t^T\!\!\langle f(s),\ u(s)\rangle ds+\int_t^T\!\!\langle u(s),\ h(s)\ d\overleftarrow{B}_s\rangle -\int_t^T\!\!\langle u(s),\ v(s)\ dW_s\rangle \\
        &+\int_t^T\!\!\!\|h(s)\|_2^2\ ds-\int_t^T\!\!\!\|v(s)\|_1^2\ ds+\gamma_0-\gamma_1,
      \end{split}
    \end{equation}
    where
    \begin{equation*}
        \begin{split}
		\gamma_0:=\mathbb{P}\textrm{-}\lim_{l\rightarrow\infty}
        \sum_{j=i+1}^{k_l}
     &-2\langle u(t_{j}^l)-u(t_{j-1}^l)-\tilde{W}_{t_{j}^l}
        +\tilde{W}_{t_{j-1}^l}
        +\tilde{B}_{t_{j}^l}
        -\tilde{B}_{t_{j-1}^l},\
                    \tilde{W}_{t_{j}^l}
                    -\tilde{W}_{t_{j-1}^l}\!\rangle , \\
		    \gamma_1:=\mathbb{P}\textrm{-}\lim_{l\rightarrow\infty}
        \sum_{j=i+1}^{k_l}
     &
        \|u(t_{j}^l)-u(t_{j-1}^l)-\tilde{W}_{t_{j}^l}
        +\tilde{W}_{t_{j-1}^l}
        +\tilde{B}_{t_{j}^l}
        -\tilde{B}_{t_{j-1}^l}\|^2
        \end{split}
    \end{equation*}
    exist and $\mathbb{P}\textrm{-}\lim$ denotes the limit in probability. Therefore, it
    remains to show that $\gamma_0=\gamma_1=0$. In a similar way to the definition of $\tilde{u}^l$ and $\bar{u}^l$, we define
    $\tilde{\tilde{W}}^l$,$\bar{\tilde{W}}^l$,$\tilde{\tilde{B}}^l$
    and $\bar{\tilde{B}}^l$.
  For each $n\in\bN$, we have
    \begin{equation*}
      \begin{split}
	      \gamma_1=\ \mathbb{P}\textrm{-}\lim_{l\rightarrow \infty}\Big(
        &\sum_{j=i+1}^{k_l}\| u(t_{j}^l)-u(t_{j-1}^l)-\tilde{W}_{t_{j}^l}
            +\tilde{W}_{t_{j-1}^l}
            +\tilde{B}_{t_{j}^l}
            -\tilde{B}_{t_{j-1}^l}\|^2\\
	    =\ \mathbb{P}\textrm{-}\lim_{l\rightarrow \infty}\bigg(
        &
            \int_t^T\!\!-\Big\langle f(s),\ \tilde{u}^l(s)-\bar{u}^l(s)
            +P_n\left(\bar{\tilde{W}}^l_s-\tilde{\tilde{W}}^l_s
            +\tilde{\tilde{B}}^l_s-\bar{\tilde{B}}^l_s\right)\Big\rangle ds         \\
        &
	 +
            \left\langle \xi-u(t_{k_l-1}^l)-\tilde{W}_{t_{k_l}^l}
            +\tilde{W}_{t_{k_l-1}^l}
            +\tilde{B}_{t_{k_l}^l}
	    -\tilde{B}_{t_{k_l-1}^l},\ \xi+P_n(\tilde{B}_T-\tilde{W}_T))\right\rangle\\
        &
            +\sum_{j=i+1}^{k_l}
            \Big\langle u(t_{j}^l)-u(t_{j-1}^l)-\tilde{W}_{t_{j}^l}
            +\tilde{W}_{t_{j-1}^l}
            +\tilde{B}_{t_{j}^l}
            -\tilde{B}_{t_{j-1}^l},\ \\
            &~~\quad
            (1-P_n)\left(-\tilde{W}_{t_{j}^l}
            +\tilde{W}_{t_{j-1}^l}
            +\tilde{B}_{t_{j}^l}
            -\tilde{B}_{t_{j-1}^l}\right)\Big\rangle
            \bigg)                                                              \\
	    =\ \mathbb{P}\textrm{-}\lim_{l\rightarrow \infty}\Big(
        &A_1+A_2+A_3\Big).
      \end{split}
    \end{equation*}
    From Lemma \ref{lem ito formula} it follows that
    $\mathbb{P}\textrm{-}\lim_{l\rightarrow\infty}\left(\int_t^T(f(s),\tilde{u}^l(s)-\bar{u}^l(s))ds\right)=0$.
     Since $u$ is weakly continuous in $H$,
     we have $\mathbb{P}-\lim_{l\rightarrow\infty}A_2=0$. Moreover, as $P_n\tilde{W}$
    and $P_n\tilde{B}$ are continuous processes in $V$,
    $$\textrm{P-}\lim_{l\rightarrow\infty}\left(\int_t^T\left\langle f(s),\
    P_n(\bar{\tilde{W}}^l-\tilde{\tilde{W}}^l(s)+
    \tilde{\tilde{B}}^l-\bar{\tilde{B}}^l(s))\right\rangle\ ds\right)=0.$$
    Thus, we have
    \begin{equation}
      \begin{split}
        \gamma_1\leq\
        &
	\mathbb{P}\textrm{-}\lim_{l\rightarrow\infty}\left(
            \sum_{j=i+1}^{k_l}\left\|u(t_{j}^l)-u(t_{j-1}^l)
            -\tilde{W}_{t_{j}^l}+\tilde{W}_{t_{j-1}^l}
            +\tilde{B}_{t_{j}^l}
            -\tilde{B}_{t_{j-1}^l}\right\|^2  \right)^{1/2}             \\
        &
            \cdot\left(\sum_{j=i+1}^{k_l}
            \left\|(1-P_n)\left(\tilde{W}_{t_{j-1}^l}
            -\tilde{W}_{t_{j}^l}
            +\tilde{B}_{t_{j}^l}
            -\tilde{B}_{t_{j-1}^l}   \right)\right\|^2
                \right)^{1/2}                                                    \\
        =\ &
        \gamma_1^{1/2}\left\langle (1-P_n)\left(-\tilde{W}+\tilde{B}\right)\right\rangle_T^{1/2}.
      \end{split}
    \end{equation}
    By Lebesgue's dominated convergence theorem, we have
    \begin{equation*}
      \begin{split}
        &
            \lim_{n\rightarrow \infty}E\left[ \left\langle (1-P_n)(\tilde{W}+\tilde{B})\right\rangle_T \right]\\
        =&\
            \lim_{n\rightarrow \infty}E\left[
            \int_0^T\left(\|(1-P_n)h(s)\|_2^2+\|(1-P_n)v(s)\|_1^2\right)\,ds \right]
            =0.
      \end{split}
    \end{equation*}
    Hence, $\gamma_1=0$.

    Similarly,
    \begin{equation*}
	    \begin{split}
	    \gamma_0=\ \mathbb{P}\textrm{-}\lim_{l\rightarrow \infty}
        &\sum_{j=i+1}^{k_l}-2\left\langle u(t_{j}^l)-u(t_{j-1}^l)-\tilde{W}_{t_{j}^l}
            +\tilde{W}_{t_{j-1}^l}
            +\tilde{B}_{t_{j}^l}
	    -\tilde{B}_{t_{j-1}^l},\ \tilde{W}_{t^l_{j}}-\tilde{W}_{t^l_{j-1}}
	    \right\rangle\\
	    =\ 2\mathbb{P}\textrm{-}\lim_{l\rightarrow \infty}
        &\bigg(
            \int_t^T\!\!\left\langle f(s),+P_n\left(\tilde{\tilde{W}}^l_s-\bar{\tilde{W}}^l_s\right)\right\rangle\ ds         \\
        &
	 +
            \left\langle \xi-u(t_{k_l-1}^l)-\tilde{W}_{t_{k_l}^l}
            +\tilde{W}_{t_{k_l-1}^l}
            +\tilde{B}_{t_{k_l}^l}
	    -\tilde{B}_{t_{k_l-1}^l},\ P_n\left(\tilde{W}_T\right)\right\rangle\\
        &
            +\sum_{j=i+1}^{k_l}
            \Big\langle u(t_{j}^l)-u(t_{j-1}^l)-\tilde{W}_{t_{j}^l}
            +\tilde{W}_{t_{j-1}^l}
            +\tilde{B}_{t_{j}^l}
            -\tilde{B}_{t_{j-1}^l},\ \\
            &~~\quad
            (1-P_n)\left(\tilde{W}_{t_{j}^l}
            -\tilde{W}_{t_{j-1}^l}
            \right)\Big\rangle  \bigg)                                                              \\
	    =\ 2\mathbb{P}\textrm{-}\lim_{l\rightarrow \infty}
	    &\bigg(
            \sum_{j=i+1}^{k_l}
            \Big\langle u(t_{j}^l)-u(t_{j-1}^l)-\tilde{W}_{t_{j}^l}
            +\tilde{W}_{t_{j-1}^l}
            +\tilde{B}_{t_{j}^l}
            -\tilde{B}_{t_{j-1}^l},\ \\
            &~~\quad
            (1-P_n)\left(\tilde{W}_{t_{j}^l}
            -\tilde{W}_{t_{j-1}^l}
            \right)\Big\rangle  \bigg)        \\
	    \leq\
        2\mathbb{P}\textrm{-}\lim_{l\rightarrow \infty}
	    &\left(
	   \sum_{j=i+1}^{k_l}
	    \left\|u(t_{j}^l)-u(t_{j-1}^l)
            -\tilde{W}_{t_{j}^l}+\tilde{W}_{t_{j-1}^l}
	    +\tilde{B}_{t_j^l}-\tilde{B}_{t_{j-1}}\right\|^2
	    \right)^{1/2}             \\
        &
            \cdot\left(\sum_{j=i+1}^{k_l}
            \left\|(1-P_n)\left(\tilde{W}_{t_{j}^l}
            -\tilde{W}_{t_{j-1}^l}
               \right)\right\|^2
                \right)^{1/2}                                                    \\
	\leq
	2\gamma_1^{1/2}\Big\langle (1-
    &P_n)\left(\tilde{W}\right) \Big\rangle_T^{1/2}.
	\end{split}
    \end{equation*}
    from which we deduce that $\gamma_0=\gamma_1=0$.

    \tbf{Step 4}.
       we prove \eqref{Ito formula square} holds for all $t\in [0,T]\backslash I$.

    In view of \tbf{Step 2}, there exists $\Omega'\in \sF$ with
    probability 1 such that both the limits in \eqref{proof of ito formula claim
    a} are point-wise ones in $\Omega'$ for some subsequence (denoted again by
     $l\rightarrow\infty$) and \eqref{Ito formula square} holds for all $t\in
     I$ on $\Omega'$. Fix $t\in [0,T]\backslash I$. In this case, for any
     $l\in \bN$ there exists a unique $j(l)>0$ such that $t\in [t_{j(l)-1}^l,t_{j(l)}^l
     [$. Letting $t(l):=t_{j(l)}^l$, we have $t(l)\downarrow t$ as $l\uparrow
     \infty$. By \tbf{Step 3}, for any $l>m$ we have
     \begin{equation}\label{proof of ito formula claim iiia}
       \begin{split}
         &\|u(t(l))-u(t(m))\|^2\\
        =&\
            2\int_{t(l)}^{t(m)}\!\!\langle f(s),\ u(s)-u(t(m))\rangle ds
            +2\int_{t(l)}^{t(m)}\!\!\langle u(s)-u(t(m)),\ h(s)d\overleftarrow{B}_s\rangle \\
         &\
            -2\int_{t(l)}^{t(m)}\!\!\langle u(s)-u(t(l)),\ v(s)dW_s\rangle
            +\langle\tilde{W}\rangle_{t(l)}-\langle\tilde{W}\rangle_{t(m)}
          \\
          &~~~~~~~-2\langle u(t(l))-u(t(m)),\ \int_{t(l)}^{t(m)}\ v(s)\ dW_s\rangle -\langle\tilde{B}\rangle_{t(l)}+\langle\tilde{B}\rangle_{t(m)}
       \end{split}
     \end{equation}
    By Lemma \ref{lem ito formula}, selecting another subsequence if
    necessary, we conclude for some $\Omega''\subset \Omega'$ with probability 1 such that
    $$\lim_{m\rightarrow \infty}\int_0^T\!\!\left|\langle f(s),\ u(s)-\tilde{u}^m(s)\rangle \right|\ ds=0.$$
    Since
    \begin{equation*}
      \begin{split}
        \sup_{l>m}\int_{t(l)}^{t(m)}\!\!\left|\langle f(s),\ u(s)-\tilde{u}^m(s)\rangle \right|\ ds
        \leq \int_0^T\!\!\left|\langle f(s),\ u(s)-\tilde{u}^m(s)\rangle \right|\ ds,
      \end{split}
    \end{equation*}
    there holds that
    $$\lim_{m\rightarrow \infty}\sup_{l>m}
    \int_{t(l)}^{t(m)}\!\!\left| \langle f(s),\ u(s)-\tilde{u}^m(s)\rangle  \right|\ ds=0$$
    on $\Omega''$.
    Moreover, as
    $$2\sup_{l>m}\left|\int_{t(l)}^{t(m)}\!\langle u(s)-u(t(m)),\ h(s)\ d\overleftarrow{B}_s\rangle \right|
    \leq
        4\sup_{t\in[0,T]}\!\!\left| \int_t^T\!\langle u(s)-\tilde{u}^m(s),\ h(s)\ d\overleftarrow{B}_s\rangle    \right|,$$
    in view of \eqref{proof of ito formula sup norm} and
      \eqref{proof of ito formula claim a} (holding pointwis on $\Omega'$) and
    by the continuity of $\langle\tilde{W}\rangle_s$, $\langle\tilde{B}\rangle_s$ and $\tilde{W}_s$,
    we conclude that
    $$\lim_{m\rightarrow \infty}\sup_{l\geq m}\|u(t(l))-u(t(m))\|^2=0$$
    holds on $\Omega''$.
    Therefore, $(u(t(l)))_{l\in\bN}$ is a Cauchy sequence in $H$ on
    $\Omega''$. As $u$ is a continuous $V'$-valued process,
    $\lim_{l\rightarrow \infty}\|u(t(l))-u(t)\|=0$ on $\Omega''$.
    Since \eqref{Ito formula square} holds for $t(l)$ on $\Omega''$,
    letting $l\rightarrow \infty$, we get \eqref{Ito formula square}
    for all $t\notin I$ on $\Omega''$.

    \tbf{Step 5}. We complete our proof by proving that
    $u\in S^2(0,T;H) $.

    From the continuity of the right-hand side
    of \eqref{Ito formula square} on $\Omega''$, it follows that the map $t\mapsto \|u(t)\|
    $ is continuous on $[0,T]$. This together with \eqref{proof of ito formula sup norm}
    and the weak continuity of
    $u(t)$ in $H$ implies $u\in S^2(0,T;H)$.
\end{proof}

\bibliographystyle{siam}

\end{document}